\documentclass[a4paper,11pt]{article}
\pdfoutput=1

\usepackage{theorem}
\usepackage[english]{babel}
\usepackage{a4}
\usepackage{epsfig}
\usepackage{amsfonts}
\usepackage{latexsym}
\usepackage{url}
\usepackage{amsmath}
\usepackage{comment}
\usepackage[table]{xcolor}

\newtheorem{theorem}{Theorem}[section]
\newtheorem{lemma}[theorem]{Lemma}
\newtheorem{proposition}[theorem]{Proposition}
\newtheorem{corollary}[theorem]{Corollary}
\newtheorem{algorithm}[theorem]{Algorithm}
\theorembodyfont{\rmfamily}
\newtheorem{definition}[theorem]{Definition}
\newtheorem{remark}[theorem]{Remark}
\newtheorem{observation}[theorem]{Observation}
\newtheorem{example}[theorem]{Example}
\theorembodyfont{\rmfamily}

\newenvironment{proof}[1][\it Proof.]{\begin{trivlist}\item[\hskip \labelsep {\bfseries #1}]}{$\Box$\end{trivlist}}

\newcommand{\N}{\ensuremath{\mathbb{N}}}
\newcommand{\Z}{\ensuremath{\mathbb{Z}}}

\newcommand{\R}{\ensuremath{\mathbb{R}}}
\newcommand{\CC}{\ensuremath{\mathbb{C}}}
\newcommand{\A}{\ensuremath{\mathcal{A}}}
\newcommand{\I}{\ensuremath{\mathcal{I}}}
\newcommand{\F}{\ensuremath{\mathcal{F}}}
\newcommand{\LL}{\ensuremath{\mathcal{L}}}

\newcommand{\mcc}{\ensuremath{C}}

\begin{document}
\title{Tropical Homotopy Continuation}
\author{Anders Nedergaard Jensen\\ Technische Universit\"at Kaiserslautern
}
\maketitle
\begin{abstract}
Inspired by numerical homotopy methods we propose a combinatorial homotopy algorithm for finding all isolated solutions to a tropical polynomial systems of $n$ tropical polynomials in $n$ variables. In particular, a tropicalisation of the numerical ``regeneration'' technique 
leads to a new method for enumerating the mixed cells of a mixed subdivision. This tropical approach shares some ideas with the recent algorithm by Malajovich. However, our algorithm has several advantages. It is  memoryless,
parallelisable as a tree traversal, exact and relies on symbolic perturbations. Our computational experiments show that the method is competitive and especially fast on the Katsura class of examples.
\end{abstract}
\section{Introduction}
In \emph{Numerical Algebraic Geometry} the main objective is to find solution components of systems of polynomial equations over the real or complex numbers by combining numerical methods with algebraic geometry arguments. The connection to polyhedral geometry is established via the BKK Theorem~\cite{bernstein} and its algorithms~\cite{Verschelde94,hubersturmfels}. The language of \emph{Tropical Geometry} conveniently allows us to express the connection: ``Given Newton polytopes, the number of intersection points of $n$ general hypersurfaces in $(\CC^*)^n$ equals the number of intersection points of $n$ generic tropical hypersurfaces in $\R^n$''. For this reason we wish to study and solve generic, square tropical systems of polynomial equations.

Similar to numerical homotopy methods, our approach will be to investigate how the solutions to a system of polynomial equations change as the coefficients change continuously. We illustrate the idea by an example. The tropical system
\begin{align}
\begin{split}(0)\odot y^{\odot 2}\oplus \textup{\framebox{(-2)}}\odot x\odot y\oplus (0)\odot x\oplus (0)\\
(-4)\odot x\odot y \oplus (-8)\odot x^{\dot 2} \oplus (-3)\odot y \oplus (0)
\end{split}
\label{eq:basicexample}
\end{align}
has the intersection of two tropical hypersurfaces in $\R^2$ as solution set (see Figure~\ref{fig:basicexample}). Here $\odot$ denotes sum and $\oplus$ maximum. By convention, for example $(x,y)=(8/3,4/3)$ is a solution to the system above because \emph{each tropical polynomial attains its maximum at at least two of its terms}. As we change the boxed coefficient to $-1$, the tropical hypersurface and intersection points change as illustrated in Figure~\ref{fig:basicexample2}
\begin{figure}
\begin{center}
\includegraphics[scale=0.57]{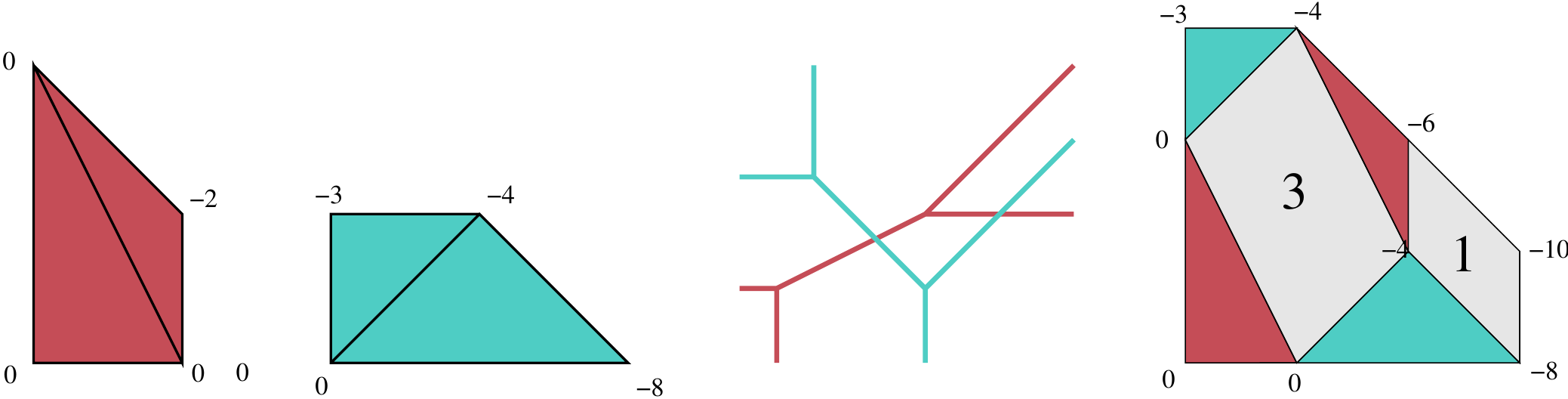}
\end{center}
\caption{The Newton polytopes of the tropical polynomials in Equation~\ref{eq:basicexample} (left). Their tropical hypersurfaces (center). The dual mixed subdivision (right).}
\label{fig:basicexample}
\end{figure}
\begin{figure}
\begin{center}
\includegraphics[scale=0.57]{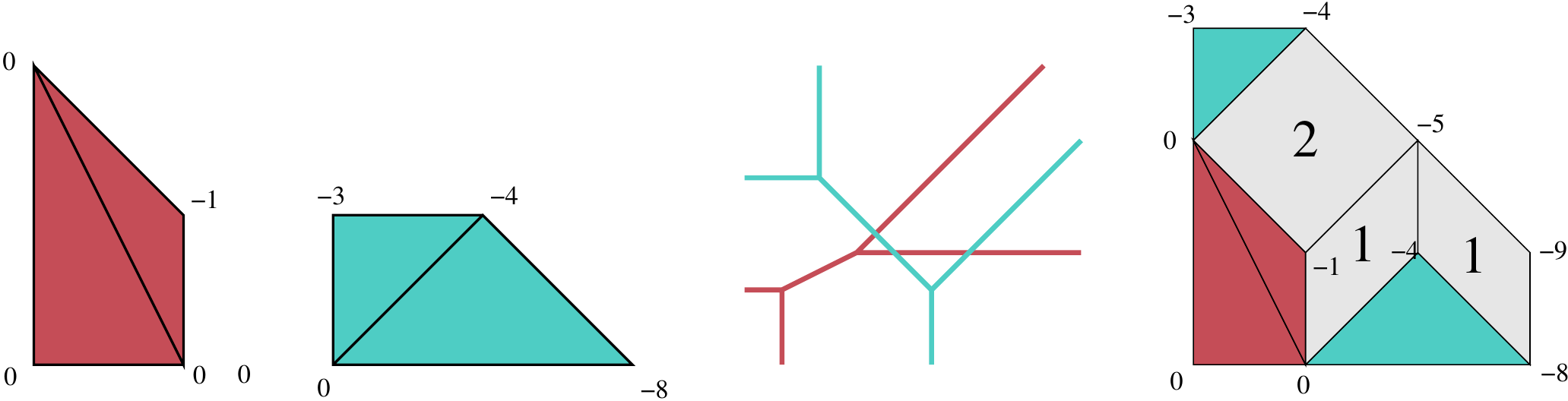}
\end{center}
\caption{The same situation as in Figure~\ref{fig:basicexample}, but with one coefficient changed.}
\label{fig:basicexample2}
\end{figure}
--- the intersection point splits into two. We propose a \emph{tropical homotopy method} which keeps track of how this happens as the coefficients vary. In particular we are interested in what happens when coefficients go to $-\infty$, as this will allow us to ``break off'' pieces of the polytopes and compute mixed volumes and mixed cells of any set of $n$ polytopes in $\R^n$.

Computing mixed volume is hard because the \#P-hard problem of computing volume of polytopes~\cite{Dyer} reduces to it~\cite{dyergritzmannhufnagel}. Nevertheless, in numerical algebraic geometry the practical problem of computing mixed volume has received much attention~\cite{DBLP:journals/dcg/VerscheldeGC96,li,Mizutani2007,leeli,li2014,malajovich}, as finding the mixed cells allows the set up of a polyhedral homotopy~\cite{hubersturmfels} with only mixed volume number of paths.

This article is structured as follows. After giving the relevant background, we study mixed cells and their cones (Section~\ref{sec:mixedcellcones}) and behaviour under bistellar flips (Section~\ref{sec:bistellar}). This leads to the tropical homotopy method (Algorithm~\ref{alg:homotopy}). We explain how it is made exact using symbolic perturbations and parallelised via reverse search \cite{af-rse-96}. In Section~\ref{sec:gensolving} and~\ref{sec:nongensolving} we show, respectively, how generic and non-generic systems are solved. In Section~\ref{sec:malajovich} we compare our approach to that of~\cite{malajovich}. Finally, we report on the implementation
and suggest future directions.

\vspace{0.5cm}
\noindent
{\bf Acknowledgements:}
The author thanks Bjarne Knudsen for providing an easy to use abstract parallel tree traversal library and Anton Leykin for many discussions about mixed volume computation and polynomial system solving. 
Inspiration also comes from earlier work with Josephine Yu on computation of tropical resultants via fan traversals.
This present work was supported by the Danish Council for Independent Research, Natural Sciences (FNU) and is now part of a project that has received funding from the European Union’s Horizon 2020 research and innovation programme under grant agreement No 676541.

\section{A numerical algebraic geometry background}
\label{sec:numericalbackground}
The main idea of this article is to tropicalise algorithms from numerical algebraic geometry (NAG). Here we present
\emph{only} the 
NAG terminology required to understand the tropical analogue and refer to~\cite{Sommese-Wampler-book-05} for a general introduction.

\begin{example}
\label{ex:sage1}
Suppose we wish to approximate the roots of the polynomial
$$f=x^3-2x^2-3x+5.$$ 
We choose a polynomial whose roots we know by construction, for illustration,
$$g=(x-1)(x-2)(x-3)$$
and set up a family of systems:
$$H(t,x)=(1-t)(x^3-2x^2-3x+5)+t(x-1)(x-2)(x-3).$$
Using Newton's method, the solutions can be tracked from $t=1$ to $t=0$ (see Figure~\ref{fig:sage1}). In general such strategy can find the \emph{isolated} solutions of \emph{square} systems, i.e.~systems $f_1=\cdots =f_n=0$ with an equal number of equations and unknowns.  The set of polynomials $\{f_1,\dots,f_n\}$ is called the \emph{target system}, $\{g_1,\dots,g_n\}$ the \emph{start system}, and $H\in\CC[t,x_1,\dots,x_n]^n$ the \emph{homotopy family}. 

\begin{figure}[b]
\begin{center}
\includegraphics[scale=0.45]{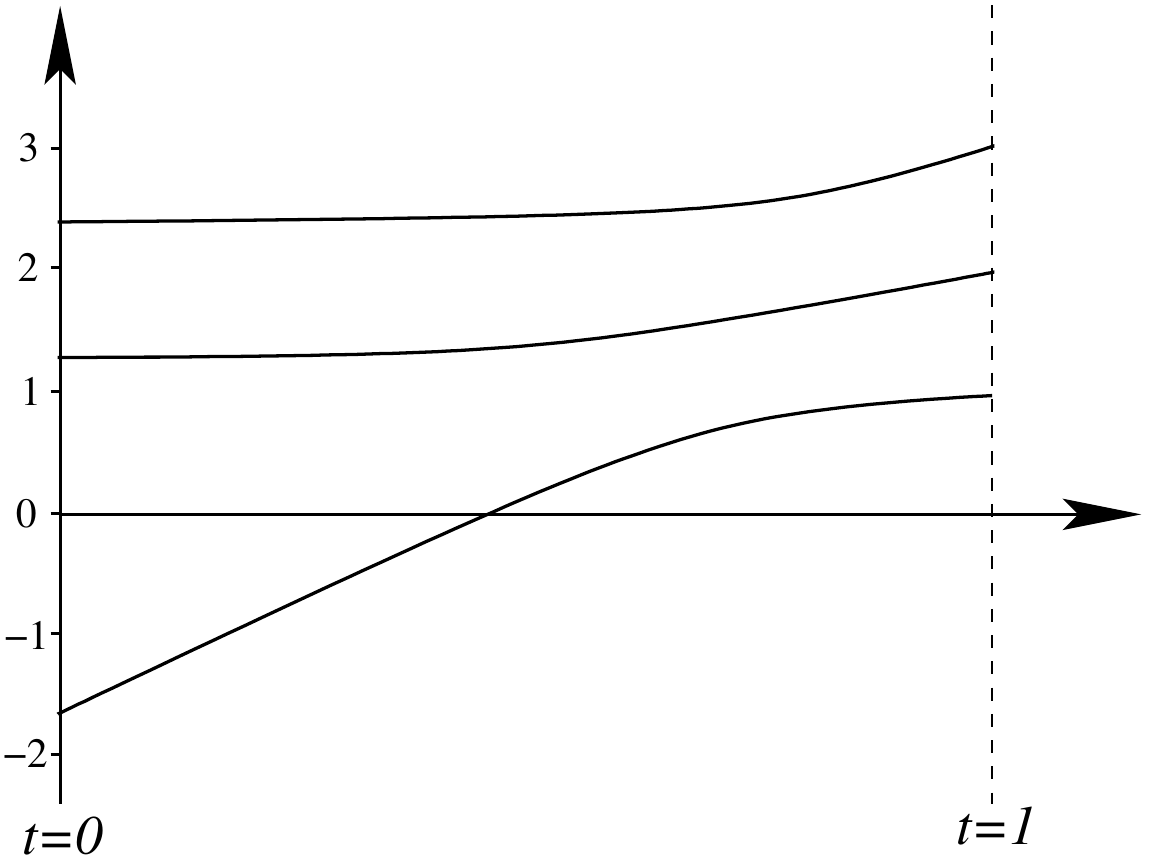}
\end{center}
\caption{The solutions in $[0,1]\times\CC$ to the system $H(t,x)=0$ in Example~\ref{ex:sage1}. Considering $t$ as a parameter, for each choice of $t$ we get three solutions. As $t$ moves from $1$ to $0$ these trace out three \emph{homotopy paths}.}
\label{fig:sage1}
\end{figure}
\end{example}

\begin{theorem}[Bezout]
For polynomials $f_1,\dots,f_n\in\CC[x_1,\dots,x_n]$ the number of solutions to $f_1=\cdots=f_n=0$ in $\CC^n$ is either infinite or bounded by $\textup{deg}(f_1)\cdots\textup{deg}(f_n)$.
\end{theorem}
Bezout's theorem leads to a particular strategy for setting up a homotopy:
\begin{definition}
Let $f_1,\dots,f_n\in\CC[x_1,\dots,x_n]$ with $\textup{deg}(f_i)=d_i$. We define the \emph{total degree homotopy family} by
$$h_i=(1-t)\cdot f_i+(t)\cdot(x_i^{d_i}-1).$$
\end{definition}
When $t=1$ we easily solve the system $h_1(x)=\cdots=h_n(x)=0$. As we let $t$ go from $1$ to $0$, the solutions of the start system change into those of the target system. In the complex plane we prefer to walk on a generic path from $t=1$ to $t=0$, as a straight line path could cause homotopy paths to  collide.

Better upper bounds for the number of solutions in the complex algebraic torus (and the number of paths to track) are obtained via mixed volumes.
\begin{definition}
Let $C_1,C_2,\dots,C_n\subseteq\R^n$ be bounded convex sets.
The function
$$f:\R_{\geq 0}^n\rightarrow\R$$
$$(\lambda_1,\dots,\lambda_n)\mapsto \textup{Volume}(\lambda_1C_1+\dots+\lambda_nC_n)$$
is polynomial in the variables $\lambda_1,\dots,\lambda_n$.
The coefficient of $\lambda_1\cdots\lambda_n$ is called the \emph{mixed volume} of $C_1,\dots,C_n$ and is denoted by $\textup{MixVol}(C_1,\dots,C_n)$.
\end{definition}

The \emph{support} $\textup{supp}(f)$ of a polynomial $f=\sum_{u}c_ux^u\in k[x_1,\dots,x_n]$ is $\{u:c_u\not=0\}$, while the \emph{Newton polytope} of $f$ is the convex hull $\textup{conv}(\textup{supp}(f))$.

\begin{theorem}[Bernstein, Khovanskii, Kushnirenko]
\label{thm:bernstein}
For $f_1,f_2,\dots,f_n\in\CC[x_1,\dots,x_n]$ the number of isolated solutions to $f_1=\cdots=f_n=0$ in $(\CC^*)^n$ is (counting multiplicities) bounded above by $\textup{MixVol}(\textup{New}(f_1),\dots,\textup{New}(f_n))$.
\end{theorem}

The polyhedral homotopy~\cite{hubersturmfels} is a realisation of Theorem~\ref{thm:bernstein}. We pick for each $u\in\textup{supp}(f_i)$ a lift $\omega_{iu}\in\R$. See Figure~\ref{fig:basicexample} (left).
Let $\widetilde{{\textup{New}(f_i)}}\subseteq \R^{n+1}$ denote the convex hull of the lifted points of $\textup{supp}(f_i)$. Viewed from above these polytopes look as in Figure~\ref{fig:basicexample} (left), while their Minkowski sum $\widetilde{{\textup{New}(f_1)}}+\cdots+\widetilde{{\textup{New}(f_n)}}$ is shown on the right as a subdivision of $\textup{New}(f_1)+\textup{New}(f_2)$. In Definition~\ref{def:mixedcell} we define the mixed cells of such a subdivision. In~\cite{hubersturmfels} mixed cells are used to find $\textup{MixVol}(\textup{New}(f_1),\dots,\textup{New}(f_n))$ start solutions near $t=0$ for the homotopy
$$h_i(t,x)=\sum_{u\in\textup{supp}(f_i)} c_{iu}t^{\omega_{iu}}x^{u}$$ 
where $t$ goes from near $0$ to $1$. This illustrates the importance of mixed cells.

In~\cite{Hauenstein:2011:RHS} the idea of \emph{regeneration} for solving systems of polynomial equations was proposed. One advantage is that no mixed cell computation is required. Another is that equations can be introduced one by one and that ``non-solution paths'' may be detected early, while only few equations are under consideration. Basically, the idea is to solve a generic linear system $l_1=\cdots=l_n=0$. After this $l_1$ is substituted by $\textup{deg}(f_1)$ random linear forms $l_{11},\dots,l_{1\textup{deg}(f_1)}$, one at a time, and the solutions of the $n$ new systems are found by making homotopy deformations from the original linear system. Now, the solutions to $l_{11}\cdot l_{12}\cdots l_{1\textup{deg}(f_1)}=l_2=\dots=l_n=0$ are known. From these the solutions of $f_1=l_2=\cdots=l_n=0$ are constructed by homotopy continuation. Successively replacing $l_2,\dots,l_n$ by $f_2,\dots,f_n$ in this way, the original system is solved.

In the following sections we tropicalise the ideas mentioned above. An interesting difference is that while genericity of the homotopy paths for numerical methods ensures that there are no collisions, this will not be the case tropically.
\section{A tropical geometry background}
\label{sec:tropicalbackground}
We will be interested in polyhedral and combinatorial aspects of tropical algebraic geometry and refer to~\cite{MaclaganSturmfels} for a general introduction. In particular we need basic definitions and results concerning tropical hypersurfaces, regular subdivisions, secondary fans and Cayley configurations. See also \cite[Section~9.2]{DLRS}.

A polynomial $f\in\R[x_1,\dots,x_n]$ can be evaluated tropically at points in $\R^n$ over the tropical semiring $(\R\cup\{-\infty\},\oplus,\odot)$ where $\oplus$ denotes maximum and $\odot$ sum. This gives a piece-wise linear function. The set of points where the maximum is attained by at least two terms of $f$ is called the \emph{tropical hypersurface} of $f$. For our purposes it will be convenient to represent $f$ by a matrix of its exponent vectors and a vector of coefficients. In Example~\ref{ex:basicexamplebbb} such representations of the polynomials in the introduction are given.
By the convex hull of a matrix, we mean the convex hull of its columns.
Our definitions go as follows.

\begin{definition}
For a matrix $A\in\Z^{n\times m}$ of exponent vectors and a coefficient vector $\omega\in\R^n$ we define
the \emph{polyhedral lift} $$L(A,\omega):=\textup{conv}_i(A_{1,i},\dots,A_{n,i},\omega_i)+\{0\}^n\times\R_{\leq 0}$$ and 
the \emph{normal complex} $$\Delta(A,w):=\pi(\textup{NF}(L(A,\omega))\wedge\{\R^n\times\{1\}\})$$ where $\textup{NF}(L(A,\omega))$ denotes the outer normal fan of $L(A,\omega)$, the wedge $\wedge$ the \emph{common refinement} (i.e. $F\wedge G:=\{f\cap g:(f,g)\in F\times G\}$) and $\pi:\R^{n+1}\rightarrow\R^n$ the projection leaving out the last coordinate.
The \emph{tropical hypersurface}
$$T(A,\omega):=\{x\in\R^n:\textup{max}_i(\omega_i+A_{1i}x_1+\cdots+A_{ni}x_n)\textup{ is attained at least twice}\}$$
is the support of a polyhedral subcomplex of the normal complex $\Delta(A,\omega)$. This subcomplex we also call the tropical hypersurface. The subdivision of the \emph{Newton polytope} $\textup{conv}_i(A_i)$ obtained by projecting the upper faces of $L(A,\omega)$ to $\R^n$ is called a \emph{regular subdivision} and is combinatorially dual to $\Delta(A,\omega)$.
\end{definition}

 For fixed $A$, each $\omega\in\R^m$ gives rise to a subdivision. A definition of the \emph{secondary fan} of $A$ is given in~\cite{DLRS}. Morally, it is the coarsest fan in $\R^m$ such that the subdivision is constant on the relative interior of each cone. However, triangulations with \emph{marked points} need to be considered for a precise definition.

For natural numbers $m_1,\dots,m_k$, $m:=\sum_im_i$, matrices $A_i\in\Z^{n\times m_i}$ and vectors $\omega_i\in\R^{m_i}$, we are interested in solving their \emph{tropical polynomial system} by which we mean finding the intersection $T(A_1,\omega_1)\cap\cdots\cap T(A_k,\omega_k)$ of tropical hypersurfaces. However, it is more convenient to study the common refinement
$$\Delta(A_1,\omega_1)\wedge\dots\wedge \Delta(A_k,\omega_k)$$
containing $T(A_1,\omega_1)\wedge\cdots \wedge T(A_k,\omega_k)$ as a subcomplex.

We will argue that $\Delta(A_1,\omega_1)\wedge\dots\wedge \Delta(A_k,\omega_k)$ depends only on the tropical hypersurface of the \emph{Cayley configuration} with coefficient vector $\omega_1\times\cdots\times\omega_k$.
\begin{definition}
For matrices $A_1,\dots,A_k$ with $A_i\in\Z^{n\times m_i}$ we define the $(n+d)\times(\sum_im_i)$ Cayley matrix
$$\textup{Cayley}(A_1,\dots,A_k):=\left(\begin{array}{cccc}
A_1 & A_2 & \cdots & A_k \\
1\dots 1 & 0\dots 0 & \cdots & 0\dots 0 \\
0\dots 0 & 1\dots 1 & \cdots & 0\dots 0 \\
\vdots  & \vdots & \ddots & 0\dots 0 \\
0\dots 0 & 0\dots 0 & \cdots & 1\dots 1  \end{array}\right).
$$
\end{definition}
\begin{lemma}
Let $A_1,\dots,A_k$ with $A_i\in\Z^{n\times m_i}$. Then
$${1\over k}\sum_i\textup{conv}(A_i)=\pi(\textup{conv}(\textup{Cayley}(A_1,\dots,A_k))\cap \R^n\times ({1\over k},\dots,{1\over k}))$$
where the sum is Minkowski sum and $\pi$ projects away the last $k$ coordinates.
\end{lemma}
We leave out the proof of the lemma, but observe along the same lines that
$${1\over k}(L(A_1,\omega_1)+\cdots+L(A_k,\omega_k))=$$
$$\pi(\textup{conv}(\textup{Cayley}({A_i\choose w^t_1},\dots,{A_k\choose w^t_k}))\cap\R^{n+1}\times({1\over k},\dots,{1\over k}))+\{0\}^n\times\R_{\leq 0}=$$
$$\pi((\textup{conv}(\textup{Cayley}({A_i\choose w^t_1},\dots,{A_k\choose w^t_k}))+\{0\}^n\times\R_{\leq 0}\times\{0\}^k)\cap\R^{n+1}\times({1\over k},\dots,{1\over k}))$$
for $\pi$ leaving out the last $k$ coordinates, and therefore the common refinement
$$\bigwedge_i\Delta(A_i,w_i)\times\{1\}=\bigwedge_i \textup{NF}(L(A_i,\omega_i))\wedge\R^n\times\{1\}=\textup{NF}(\sum_iL(A_i,\omega_i))\wedge\R^n\times\{1\}$$
only depends on 
$$\textup{conv}(\textup{Cayley}({A_i\choose w^t_1},\dots,{A_k\choose w^t_k}))+\{0\}^n\times\R_{\leq 0}\times\{0\}^k$$
which is determined by $\Delta(\textup{Cayley}(A_1,\dots,A_k),\omega_1\times \cdots \times \omega_k)$.

Projecting the upper faces of $L(A,\omega)$ gives a subdivision of the convex hull of the columns of $A$. In the case of a refinement $\Delta(A_1,\omega_1)\wedge\cdots\wedge\Delta(A_k,\omega_k)$, projecting the upper faces of $L(\textup{Cayley}(A_1,\dots,A_k),\omega_1\times \cdots \times \omega_k)$ gives a subdivision of the Cayley configuration or, by intersecting with $\R^n\times\{({1\over k},\dots,{1\over k})\}$, projecting and scaling, a mixed subdivision of $\sum_i\textup{conv}(A_i)$. Cells in the mixed subdivision of $\sum_i\textup{conv}(A_i)$ arising from cells of $\textup{Cayley}(A_1,\dots,A_k)$ containing at least two columns from each $A_i$ are called \emph{fully mixed}, or just \emph{mixed} for short.

\begin{example}
\label{ex:basicexamplebbb}
Consider the matrices
$$A_1=\left(\begin{array}{cccc}
0&0&1&1\\
0&2&0&1  \end{array}\right) \textup{ and } A_2=\left(\begin{array}{cccc}
0&0&1&2\\
0&1&1&0  \end{array}\right).$$
Choosing $w_1=(0,0,0,-2)^t$ and $w_2=(0,-3,-4,-8)^t$ we get the two tropical hypersurfaces shown in the middle picture in Figure~\ref{fig:basicexample}. They are combinatorially dual to the regular subdivisions on the left, while their overlay is combinatorially dual to the mixed regular subdivision of $\textup{conv}(A_1)+\textup{conv}(A_2)$ shown on the right. The mixed subdivision has two mixed cells of area 3 and 1, respectively.
\end{example}

\section{Mixed cell cones}
\label{sec:mixedcellcones}
From now on we let $k=n$, consider a fixed tuple $\A=(A_1,\dots,A_n)$ and let $\omega\in\R^m=\R^{m_1}\times\cdots\times\R^{m_n}$ vary. For any particular choice of $\omega$, the overlay $\Delta(A_1,\omega_1)\wedge\dots\wedge \Delta(A_n,\omega_n)$ is dual to a mixed subdivision or, equivalently, a regular subdivision of $\textup{Cayley}(A)$. Therefore all possible combinatorial types of the overlay are obtained by considering all cones of the secondary fan~\cite{DLRS} of the Cayley configuration.
Ideally, we would like to do the equivalent of a Gr\"obner walk in this fan i.e. update the subdivision as $\omega$ is moved along a straight line, but because triangulations and secondary fans can be extremely large, we will consider each mixed cell independently, while having the secondary fan in mind.

It is a well known that for $\omega$ chosen generically, the induced mixed cells of $\sum_i\textup{conv}(A_i)$ have volume summing to $\textup{MixVol}(\textup{conv}(A_1),\dots,\textup{conv}(A_n))$, see \cite[Theorem~1.3.4]{DLRS}. As a corollary we get the tropical BKK theorem:
\begin{theorem}
Given a tuple $A$, for generic choices of $\omega\in\R^m$, the intersection of tropical hypersurfaces $T(A_1,\omega_1)\cap\cdots\cap T(A_n,\omega_n)$ is finite. Counted with multiplicity the intersection has cardinality $\textup{MixVol}(\textup{conv}(A_1),\dots,\textup{conv}(A_n))$.
\end{theorem}
Here the multiplicity of an intersection point is the volume of its dual mixed cell in $\sum_i\textup{conv}(A_i)$. See~\cite[Theorem 4.6.8]{MaclaganSturmfels} for a different version of this theorem.

We now extend the notion of a mixed cells. For convenience we allow an index for a column of the Cayley matrix to also index the associated column of $A_i$. This will lead to no confusion.
\begin{definition}
\label{def:mixedcell}
Given $A$ and $\omega$, a tuple $((a_1,b_1),\dots,(a_n,b_n))$ of pairs of indices to columns of $\textup{Cayley}(A)$ is called a \emph{mixed cell} if:
\begin{itemize}
\item the square submatrix of $\textup{Cayley}(A)$ consisting of columns indexed by $a_1,b_1,\dots,a_n,b_n$ has full rank, and
\item the parallelepiped $$\sum_i\textup{conv}((A_{a_i},\omega_{a_i}),(A_{b_i},\omega_{b_i}))$$ is a facet of $L(A_1,\omega_1)+\cdots+L(A_n,\omega_n)$.
\end{itemize}
\end{definition}
Given $A$, we call a tuple $((a_1,b_1),\dots,(a_n,b_n))$ a \emph{mixed cell candidate} if it satisfies the first condition of Definition~\ref{def:mixedcell}.

\begin{definition}
Given a tuple $A$ of configurations and a mixed cell candidate $M=((a_1,b_1),\dots,(a_n,b_n))$, we define the \emph{mixed cell cone} as the set
$$\mcc_M:=\overline{\{\omega\in\R^n:((a_1,b_1),\dots,(a_n,b_n))\textup{ is a mixed cell for } (A,\omega)\}}$$
where the closure is taken in the Euclidean topology.
\end{definition}

\begin{lemma}
\label{lem:ineq}
The mixed cell cone $C$ of a candidate $M=((a_1,b_1),\dots,(a_n,b_n))$ is described by $\sum_i(m_i-2)$ irredundant linear inequalities. The vector of coefficients for each of these inequalities is a circuit of the Cayley matrix.
\end{lemma}
\begin{proof}
For a given lift $\omega\in\R^m$, the parallelepiped in question has a normal $p\in \R^n\times\R_{>0}$, unique up to scaling. By $\textup{face}_p(P)$ of a polytope $P$ we mean the face of $P$ where the maximum of the dot product with $p$ is attained. The candidate $((a_1,b_1),\dots,(a_n,b_n))$ is a mixed cell if and only if for all $i$:
$$\textup{face}_p(\textup{conv}_j((A_i)_{j},(\omega_i)_{j}))=\textup{conv}((A_{a_i},\omega_{a_i}),(A_{b_i},\omega_{b_i})).$$
Equivalently, if and only if all columns of $A_i$ not indexed by $a_i$ and $b_i$ are lifted lower than the hyperplane with normal $p$ passing through $(A_{a_i},\omega_{a_i})$. For each $i$ this is a sign condition on the determinant of the corresponding $(2n+1)\times(2n+1)$-submatrix of $\textup{Cayley}(\A)$ with $\omega$ appended as a row. Because any coordinate not mentioned in $((a_1,b_1),\dots,(a_n,b_n))$ appears with non-zero coefficient in exactly one inequality, each such inequality defines facets of $C$.

To see that each inequality comes from a circuit, recall that the square submatrix of $\textup{Cayley}(\A)$ indexed by $\{a_1,b_1,\dots,a_n,b_n\}$ has nullity $0$. Therefore, after appending one column it has nullity $1$. Hence all non-zero elements of the null space of the $2n\times(2n+1)$ matrix have the same support.
\end{proof}
\begin{observation}
\label{obs:circuitsign}
In the proof above, the inequality $\omega\cdot c\geq 0$ arising from considering index $\gamma$ in the complement of the mixed cell has $c_\gamma<0$ because low lifts of the $\gamma$th coordinate of $\omega$ are allowed in the mixed cell cone.
\end{observation}
\begin{example}
In Figure~\ref{fig:basicexample} a mixed cell of volume $3$ appears. Its cone is described by $(4-2)+(4-2)=4$ inequalities. Each of these can be obtained by considering $4\times 5$ submatrices $A$ of the Cayley configuration 

$$\left(\begin{array}
{c>{\columncolor{yellow!20}}c>{\columncolor{yellow!20}}c>{\columncolor{olive!30}}c>{\columncolor{yellow!20}}cc>{\columncolor{yellow!20}}cc}
0&0&1&1&0&0&1&2\\
0&2&0&1&0&1&1&0\\
1&1&1&1&0&0&0&0\\
0&0&0&0&1&1&1&1
  \end{array}\right) $$
involving the $4$ columns of the cell and one additional column. Each constraint is obtained as $\omega\cdot c \geq 0$ by choosing a non-zero $c\in\textup{NullSpace}(A)$ with the entry indexed by the additional column being negative.
For the $4$th column in the first configuration the inequality becomes $(0,1,2,-3,-1,0,1,0)\cdot \omega\geq 0$.
\end{example}

\section{Mixed cell behaviour under bistellar flips}
\label{sec:bistellar}

Suppose we are given a generic lift $\omega$ such that the mixed cells all arise from a mixed cell candidate. By Lemma~\ref{lem:ineq} the closure $C$ of vectors giving exactly these mixed cells is an intersection of polyhedral cones given by a certain set of linear inequalities. Note that it is possible that some of the inequalities obtained from Lemma~\ref{lem:ineq} are redundant for $C$. We now investigate what happens to the mixed cells as $\omega$ passes through the relative interior of a facet of $C$.

\begin{example}
Figure~\ref{fig:basicexample} (right) shows a mixed regular subdivision whose Cayley triangulation $A$ has $6$ maximal simplices. In Figure~\ref{fig:basicexample2} the corresponding triangulation $B$ has $7$ maximal simplices. When going from Figure~\ref{fig:basicexample} to Figure~\ref{fig:basicexample2} we pass through the hyperplane given by $(0,1,2,-3,-1,0,1,0)\cdot \omega= 0$. The Cayley subconfiguration indexed by the support of this equation contains $5$ vectors and has two regular triangulations (with indices referring to $\textup{Caley}(A_1,A_2)$):
\begin{itemize}
\item $A'=\{\{2,3,4,7\},\{2,3,5,7\}\}$ and
\item $B'=\{\{2,3,4,5\},\{2,4,5,7\},\{3,4,5,7\}\}$.
\end{itemize}
The rule for passing from $A$ to $B$ is:
\begin{equation}
\label{eq:rule}
 B=\{\sigma\in A:\not\exists\tau\in A':\tau\subseteq\sigma\}\cup\{\sigma\setminus\tau\cup\rho:\sigma\in A\wedge\rho\in B'\wedge A'\ni\tau\subseteq\sigma\}.
\end{equation}
\end{example}

\begin{lemma}
\label{lem:abprimedef}
Let $C$ be a mixed cell cone for a configuration $\A$, $c$ a circuit defining a facet $F$ of $C$, $\omega\in F$ a generic point (i.e. not in the codimension 2 skeleton of the secondary fan of $\textup{Cayley}(\A)$). Define $D=\textup{supp}(c)$, $A'$ as the $\omega+\varepsilon c$ induced regular triangulation of $\textup{Cayley}(\A)$ restricted to $D$ and $B'$ as the triangulation induced by $\omega-\varepsilon c$ for $\epsilon>0$ sufficiently small. Then
\begin{itemize}
\item $A'=\{D\setminus\{d\}:d\in D\wedge c_d<0\}$ and
\item $B'=\{D\setminus\{d\}:d\in D\wedge c_d>0\}$.
\end{itemize}
\end{lemma}
\begin{proof}
The condition on a lift $l$ to lift the vectors of the $2n\times |D|$ submatrix of $\textup{Cayley}(\A)$ to a non-vertical hyperplane is that the $(2n+1)\times |D|$ submatrix with the row $l$ appended does not have full column rank. Because $c$ is in the nullspace of the $2n\times |D|$ matrix, $l\cdot c=0$ is an equivalent condition.  We know that $l\cdot c>0$ for $l=\omega+\varepsilon c$ with $\varepsilon>0$. The condition for $D\setminus\{d\}$ to appear in $A'$ (that $d$th column is lifted low) becomes that increasing $l_d$ eventually causes an equality. Equivalently that $c_d<0$. A similar argument applies to $B'$.
\end{proof}
We note that the following degenerate situation is possible.
\begin{example}
\label{ex:degenerate}
We have the Cayley configuration
$$\left(\begin{array}{ccccc}
0&1&0&1&0\\
0&0&1&0&1\\
1&1&1&0&0\\
0&0&0&1&1  \end{array}\right)$$
and consider the mixed cell $\{1,3,4,5\}$ and $\gamma=2$. The circuit then becomes $(0,-1,1,1,-1)$ and 
$A'=\{\{2,3,4\},\{3,4,5\}\}$ and $B'=\{\{2,4,5\},\{2,3,5\}\}$.
\end{example}
The $2n+1-|D|$ \emph{missing} elements from $\textup{supp}(c)=D$ is possibly one from $\{a_i,b_i\}$ together with pairs of vectors from the same $A_j$.
We will call an element in $A'$ resp. $B'$ \emph{mixed}, if it together with the missed vectors does not index 3 columns from the same $A_j$. In Example~\ref{ex:degenerate} above the set of indices of the missed vectors is $\{1\}$. This means that $\{2,4,5\}$ and $\{3,4,5\}$ are the only mixed sets in $A'\cup B'$.
\begin{lemma}
\label{lem:mixflip}
$A'\cup B'$ has either 2 or 3 three mixed elements. Each of $A'$ and $B'$ has one or two mixed elements.
\end{lemma}
\begin{proof}
By Lemma~\ref{lem:abprimedef} $A'\cap B'=\emptyset$. Let $i$ be the index of the configuration for which an additional column was considered to form the circuit. To pick a mixed cell from the support, for all but the $i$th configuration, we have only one choice for picking subsets. For the $i$th configuration, there are three elements to choose from. This gives three different ways of picking a cardinality 2 subset. Of these it is, by the assumption that the circuit arose from a mixed cell, never possible that $A'$ has no mixed cell and because the mixed volume is invariant it is also not possible to have no mixed cell in $B'$.
\end{proof}

\begin{lemma}
If a mixed cell $M$ of $A$ is changed when passing to $B$, then $c$ also appears as a facet inequality of the mixed cell cone $C_M$.
\end{lemma}
\begin{proof}
Because $M$ is changed, it must contain a mixed subset from $A'$. Such subset has form $D\setminus\{d\}$. But now, if we tried to determine the mixed cell cone of $M$, when we construct the circuit inequality induced by letting $\gamma=d$, we would indeed obtain $c$ as a circuit. By Lemma~\ref{lem:ineq} the circuit defines a facet.
\end{proof}
This lemma will be important later for parallelisation. We will say that $M$ \emph{gives rise to} $c$ and we could write up a combinatorial condition for this. However, in our final tree traversing algorithm deciding if $M$ gives rise to $c$ is not necessary.

A consequence of Lemma~\ref{lem:mixflip} is that the mixed cells in $B$ can be reconstructed from the mixed cells in $A$. This gives the following simplified algorithm for keeping track of only the mixed cells when applying Equation~(\ref{eq:rule}).
\begin{algorithm}[Bistellar flip for mixed cells]$ $\label{alg:bistellar}\\
{\bf Input:}
\begin{itemize}
\item A tuple $\A=(A_1,\dots,A_n)$.
\item A mixed cell $M=((a_1,b_1),\dots,(a_n,b_n))$ with mixed cell cone $C_M$.
\item A circuit $c\in\R^m$ defining a facet $F$ of $C_M$ with $\omega\in F$ being generic.
\item The set $A_\textup{mix}$ of all mixed cells w.r.t. $\omega+\varepsilon c$ for $\varepsilon>0$ sufficiently small.
\end{itemize}
{\bf Output:} The mixed cells $B_\textup{mix}$ w.r.t. $\omega-\varepsilon c$ for $\varepsilon>0$ sufficiently small.
\begin{itemize}
\item $B_\textup{mix}:=\emptyset$.
\item Let $i$ be the configuration and $\gamma$ the column index giving rise to $c$.
\item Let $\alpha=a_i$, $\beta=b_i$.
\item For each mixed cell $\sigma\in A_\textup{mix}$
\begin{itemize}
\item If $\sigma$ gives rise to $c$
\begin{itemize}
\item if $c_\alpha>0$ then $B_\textup{mix}:=B_\textup{mix}\cup\{\sigma\cup\{\gamma\}\setminus \{\alpha\}\}$.
\item if $c_\beta>0$ then $B_\textup{mix}:=B_\textup{mix}\cup\{\sigma\cup\{\gamma\}\setminus \{\beta\}\}$.
\end{itemize}

\end{itemize}
\item Return $B_\textup{mix}$
\end{itemize}
\end{algorithm}
\begin{remark}
\label{rem:merge}
It is possible that the same set is inserted into $B$ more than once. This happens if two mixed cells of $A$ are the same, except on $A_i$  --- for example if we reverse our running example i.e. go from Figure~\ref{fig:basicexample2} to Figure~\ref{fig:basicexample}.
\end{remark}
\begin{remark}
Recall that it is possible that the subconfiguration indexed by the support of the circuit did not involve all $n$ configurations. In this case maximal cells would just keep their indices to irrelevant configurations as they change. This will be essential later: If two different mixed cells gave rise to the same circuit wall via Lemma~\ref{lem:ineq} and if the mixed cells are the same restricted to the support of the circuit then the two mixed cells will give rise to two disjoint set of mixed cells on the other side of the wall.
\end{remark}
In Section~\ref{sec:reversesearch} we will see that these two observations together with~\ref{lem:mixflip} allow us to treat mixed cells independently and therefore in parallel.

\section{Tropical homotopy continuation}
In this section we present our main algorithm. If we know the mixed cells of $\A$ with respect to a lift $\omega\in\R^m$ and wish to obtain the mixed cells with respect to some vector $\omega'\in\R^m$, the idea is to continuously move from $\omega$ to $\omega'$ along a straight line $l$. When a facet of a mixed cell cone of one of our known mixed cells is reached, we perform the bistellar flip on the mixed cells via Algorithm~\ref{alg:bistellar}.

For the description to make sense it is important that $l$ indeed does exit the intersection of the mixed cell cones in the relative interior of a facet, and not a lower dimensional face. Following the ideas of \emph{computational geometry} we solve the problem by assuming that the start vector $\omega$ is in general position w.r.t. $\A$. We will be more precise in Section~\ref{sec:generic}
\begin{algorithm}[Tropical homotopy continuation]$ $\label{alg:homotopy}\\
{\bf Input:}
\begin{itemize}
\item A tuple $\A=(A_1,\dots,A_k)$ with $A_i\in\Z^{n\times m_i}$ and $m=\sum_im_i$.
\item A generic start vector $\sigma\in\R^m$ and a target vector $\tau\in\R^m$.
\item The set $S\subseteq (\{1,\dots,m\}\times\{1,\dots,m\})^n$ of the mixed cells in the mixed subdivision of $A$ induced by $\sigma$.
\end{itemize}
{\bf Output:} The set $S'\subseteq (\{1,\dots,m\}\times\{1,\dots,m\})^n$ of mixed cells in the mixed subdivision of $A$ induced by $\tau+\varepsilon\sigma$ for $\varepsilon>0$ sufficiently small.
\begin{itemize}
\item While $\tau\not\in \bigcap_{M\in S}\mcc_M$
\begin{itemize}
\item Let $c$ be the first inequality in the description of $\bigcap_{M\in S}\mcc_M$ violated along the line from $\sigma$ to $\tau$, but satisfied at $\sigma$. \item Apply Algorithm~\ref{alg:bistellar} to update $S$ with respect to $c$.
\end{itemize}
\item Return $S':=S$.
\end{itemize}
\end{algorithm}
\subsection{Generic start vectors}
\label{sec:generic}
We will now explain how to use symbolic perturbations to ensure that $l$ passes through a unique facet when leaving a finite intersection of mixed cell cones. Since our coefficient changing strategy is similar to a Gr\"obner walk, we will use the language of term orderings to be precise about the perturbations. We, however, only need that term orderings are total orderings extending to $\R^m$ and respecting translation. By Robbiano's theorem every such ordering is also represented by a vector $a_1+\varepsilon a_2+\cdots +\varepsilon^{m-1}a_m$ with $a_1,\dots,a_m\in\R^m$ and $\varepsilon$ an infinitesimal or the parameter in the ordered field $\R(\varepsilon)$. Our strategy here is a variant of the simplification~\cite{phdthesis} of the generic Gr\"obner walk~\cite{genericgroebnerwalk}.

The starting point for $l$ will be $\sigma_\varepsilon=a_1+\varepsilon a_2+\cdots +\varepsilon^{m-1}a_m$ and the target $\tau$. We interpolate linearly by letting $\omega(t)=(1-t)\sigma_\varepsilon+t\tau$. Let $c,c'\in\R^m$ be normals of facets of mixed cell cones with $\langle\sigma_\varepsilon,c\rangle>0$, $\langle\tau,c\rangle\geq 0$ and similar inequalities holding for $c'$ with $\langle\cdot,\cdot\rangle$ denoting the usual inner product. We may assume that $\langle\tau,c\rangle> 0$ and $\langle\tau,c'\rangle>0$, as Algorithm~\ref{alg:homotopy} does not pass through facets where a zero value is attained.

 Let $t$ and $t'$ be the values such that $\langle \omega(t),c\rangle=0$ and $\langle \omega(t'),c'\rangle=0$. Equivalently, $t={\langle\sigma_\varepsilon,c\rangle\over \langle\sigma_\varepsilon,c\rangle-\langle\tau,c\rangle}$ and $t'={\langle\sigma_\varepsilon,c'\rangle\over \langle\sigma_\varepsilon,c'\rangle-\langle\tau,c'\rangle}$. Now
\begin{align}
\begin{split}
t&<t'\Leftrightarrow\\
{\langle\sigma_\varepsilon,c\rangle\over \langle\sigma_\varepsilon,c\rangle-\langle\tau,c\rangle}&<{\langle\sigma_\varepsilon,c'\rangle\over \langle\sigma_\varepsilon,c'\rangle-\langle\tau,c'\rangle}\Leftrightarrow\\
{\langle\sigma_\varepsilon,c\rangle-\langle\tau,c\rangle\over \langle\sigma_\varepsilon,c\rangle}&>{\langle\sigma_\varepsilon,c'\rangle-\langle\tau,c'\rangle\over \langle\sigma_\varepsilon,c'\rangle}\Leftrightarrow\\
{\langle\tau,c\rangle\over \langle\sigma_\varepsilon,c\rangle}&<{\langle\tau,c'\rangle\over \langle\sigma_\varepsilon,c'\rangle}\Leftrightarrow\\
{\langle\sigma_\varepsilon,c\rangle\over \langle\tau,c\rangle}&>{\langle\sigma_\varepsilon,c'\rangle\over \langle\tau,c'\rangle}\Leftrightarrow\\
{\langle\sigma_\varepsilon, \langle\tau,c'\rangle c\rangle}&>{\langle\sigma_\varepsilon,\langle\tau,c\rangle c'\rangle}.
\end{split}
\end{align}
The last inequality is decided by comparing $\langle\tau,c'\rangle c$ and $\langle\tau,c\rangle c'$ in the ordering that $a_1,\dots,a_m$ represent and is therefore independent of the value of $\varepsilon$ for $\varepsilon>0$ sufficiently small. Consequently, we can decide which circuit hyperplane our perturbed line first meets without computing in $\R(\varepsilon)$, but instead comparing vectors arising from $\tau$ and the circuits in the ordering $\prec$.
\subsection{Reverse search and parallelisation}
\label{sec:reversesearch}
The loop in Algorithm~\ref{alg:bistellar} processes each mixed cell independently with the exception that the set $B_\textup{mix}$ is modified for all mixed cells. Rather than updating $B_\textup{mix}$ we could however recursively continue the processing of the next cell. This will give an algorithm that computes all target solutions i.e. mixed cells as the leaves of a recursion tree. It, however, has the problem that the same mixed cells may be computed more than once, as tropical homotopy paths can merge (Remark~\ref{rem:merge}). We solve this problem by applying \emph{reverse search}~\cite{af-rse-96}.

Imagine the union of homotopy paths as an embedded graph $G$ in $[0,1]\times\R^{n}$. Each edge of $G$ is oriented in direction of decreasing $t$. In particular $G$ has no directed cycles. To do reverse search on $G$, we invent a rule for which ingoing edge to keep at each vertex of $G$. This turns $G$ into a directed forest $F$ with roots only placed at $t=1$.  All vertices of $G$ can now be found by traversing all trees in $F$ starting at $t=1$.

The point on an edge in $G$ is obtained as a linear function in the lift $\omega$. That is, as a linear function of $t$, since it is the solution $x\in\R^n$ to the system $$ 
\left(\begin{array}{c}
x\\
1  \end{array}\right)^t\cdot
\left(\begin{array}{c}
\textup{Cayley}(A_1,\dots,A_n)\\
w  \end{array}\right)\cdot
\left(\begin{array}{ccc}
1 & \hdots & 0\\
-1 & \hdots & 0\\
\vdots & \ddots & \vdots\\
0 & \hdots & 1\\
0 & \hdots & -1\\
 \end{array}\right)
=0.
$$

\begin{observation}
Extending on Observation~\ref{obs:circuitsign},
consider an edge
parametri\-sed by the mixed cell $((a_1,b_1),\dots,(a_n,b_n))$. Let $\gamma$ be an entry from the $i$th configuration giving rise to a circuit $c$. Assume that it is possible to choose the sign of $c$ such that $\langle\sigma_\varepsilon,c\rangle>0$ and $\langle\tau,c\rangle\geq 0$ (meaning $c$ is an inner normal of the mixed cell cone of $((a_1,b_1),\dots,(a_n,b_n))$). By Observation~\ref{obs:circuitsign} $c_\gamma<0$. The edge is oriented towards the vertex $v$ associated to with the selection $((a_1,b_1),\dots,(a_i,b_i,\gamma),\dots,(a_n,b_n))$. The edge associated with $((a_1,b_1),\dots,(a_i,\gamma),\dots,(a_n,b_n))$ is incident to $v$ and oriented towards $v$ if and only if $c_{b_i}<0$. Finally the edge associated with $((a_1,b_1),\dots,(b_i,\gamma),\dots,(a_n,b_n))$ is incident to $v$ and oriented towards $v$ if and only if $c_{a_i}<0$.
\end{observation}

A common problem in reverse search is that each vertex needs to be processed more than once, namely once for each of its ingoing edges in $G$, since one has to determine information about the end vertex of an edge in $G$ to decide if the edge is also present in the tree. This also applies to reverse search for tropical homotopy. In Figure~\ref{fig:reversesearch} we depict first $G$, the tree, then all edges (mixed cell) candidates under consideration and finnally $G$ annotated with circuit signs.

\begin{figure}[b]
\begin{center}
\includegraphics[scale=0.65]{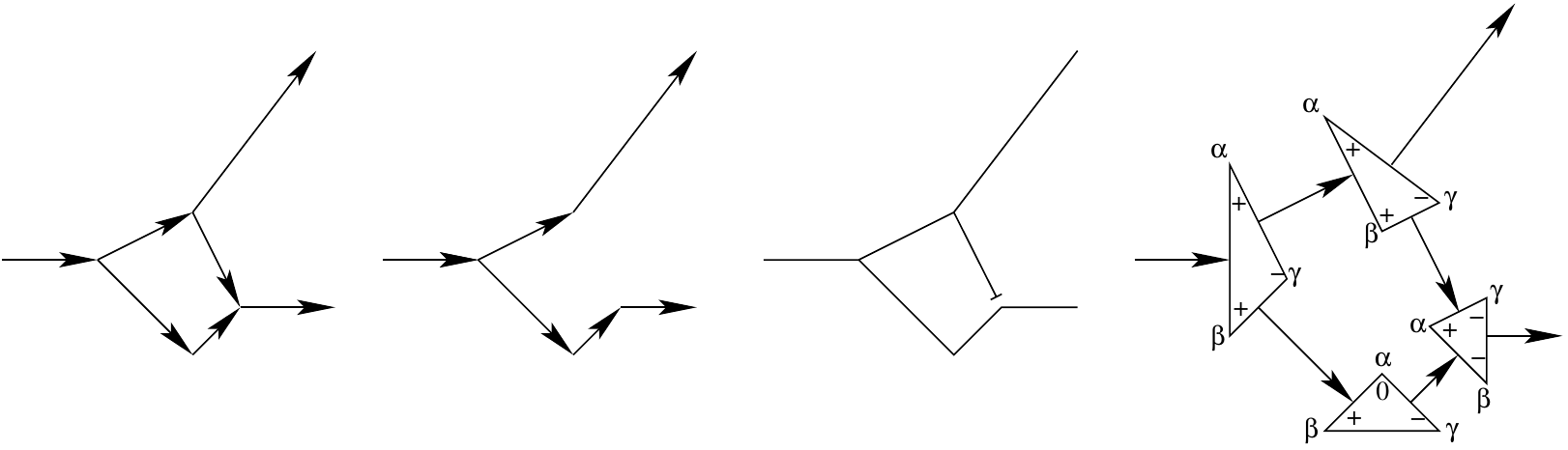}
\end{center}
\caption{The embedded directed graph $G$, the reverse search tree, the actual edges (mixed cells) under consideration and a schematic drawing of the graph with associated symbols to which the expanded reverse search rule is applied.}
\label{fig:reversesearch}
\end{figure}

To define the reverse search tree we come up with the following rule: for each vertex, i.e. a selection $((a_1,b_1),\dots,(a_i,b_i,\gamma),\dots,(a_n,b_n))$ if it has two in-edges gotten by choosing mixed cells $(\alpha,\beta)$ and $(\alpha,\gamma)$ respectively, we keep the keep the edge with smallest second index. That is, we keep $(\alpha,\beta)$ if $\beta<\gamma$.

A recursive algorithm for traversing the tree is easily described by simply replacing the processing of $\sigma$ in Algorithm~\ref{alg:bistellar} by the following and at the same time stating, that now it is impossible for a mixed cell to be inserted twice into $B_\textup{mix}$ and therefore the algorithm might as well be implemented recursively. 
\begin{itemize}
\item If $\sigma$ gives rise to $c$
\begin{itemize}
\item if $c_\alpha>0\wedge c_\beta>0$ then $B_\textup{mix}:=B_\textup{mix}\cup\{\sigma\cup\{\gamma\}\setminus \{\alpha\}\}\cup\{\sigma\cup\{\gamma\}\setminus \{\beta\}\}$
\item if $c_\alpha>0\wedge c_\beta=0$ then $B_\textup{mix}:=B_\textup{mix}\cup\{\sigma\cup\{\gamma\}\setminus \{\alpha\}\}$
\item if $c_\alpha>0\wedge c_\beta<0\wedge\beta<\gamma $ then $B_\textup{mix}:=B_\textup{mix}\cup\{\sigma\cup\{\gamma\}\setminus \{\alpha\}\}$
\item if $c_\alpha=0\wedge c_\beta>0$ then $B_\textup{mix}:=B_\textup{mix}\cup\{\sigma\cup\{\gamma\}\setminus \{\beta\}\}$

\item if $c_\alpha<0\wedge c_\beta>0\wedge\alpha<\gamma $ then $B_\textup{mix}:=B_\textup{mix}\cup\{\sigma\cup\{\gamma\}\setminus \{\beta\}\}$
\end{itemize}
\end{itemize}
The other four possibilities for signs do not appear since at least one of $c_\alpha,c_\beta$ and $c_\gamma$ must be positive for $c$ to be a circuit for the Cayley matrix.

\begin{example}
In the schematic example in Figure~\ref{fig:reversesearch}, processing the edge leaving the root in the fourth picture, we discover a circuit arising from a certain choice of $\gamma$. The signs of the entries of the circuit are shown in the left triangle. The situation is as described in the first situation in the list, the current edge exists in the tree and we have now two more mixed cells to consider.  The situation is the same for the upper edge in the next level. When processing the lower edge the signs fall into the fourth case of the list. Consequently, we were right to follow the edge and there is one more mixed cell (edge) to consider. When considering this edge, we obtain signs $c_\alpha>0, c_\beta<0$ and $c_\gamma<0$ for a circuit and item 3 tells us, assuming that $\beta<\gamma$ to also consider the last edge. When processing the triangle from the upper side, however, the names of $\gamma$ and $\beta$ are swapped, and item 3 will tell us not to follow the last edge this time.
\end{example}

Focusing on a single tree, reverse search allows us to either make a memoryless traversal of the tree or to apply a general parallel tree traversal algorithm. 
\section{Solving generic systems}
\label{sec:gensolving}
Having established the tropical homotopy algorithm we are able to solve tropical systems for any particular generic choice of coefficients if the solutions for another set of coefficients are known. 
In this section we show how to find such other set of solutions.
Our idea will be to break off pieces of larger polytopes with known mixed cells.
The following lemma is essential. For $\I\subseteq\{1,\dots,m\}$, let $\A_{I^c}$ denote the configuration restricted to columns indexed by $\{1,\dots,m\}\setminus\I$.
\begin{lemma}
\label{lem:breakingoff}
Let $\I\subseteq\{1,\dots,m\}$ and $S$ be the set of mixed cells of $\A$ with respect to $-1_\I+\varepsilon\omega$ where $1_\I$ is the characteristic vector of $\I$, $\omega\in\R^m$ and $\varepsilon>0$ is sufficiently small. Let $\pi:\R^m\rightarrow\R^{m-|\I|}$ be the projection forgetting coordinates indexed by $\I$. Then the mixed cells of $\A_{\I^c}$ induced by $\pi(\omega)$ are $\{M\in S:M\cap\I=\emptyset\}$.
\end{lemma}
\begin{proof} 
``$\supseteq$'': We will check that for $M\in S$ with $M\cap\I=\emptyset$ the vector $\pi(\omega)$ is in the mixed cell cone of $M$. For every $\gamma\in\{1,\dots,m\}\setminus\I\setminus M$ there is one inequality to check. It has form $\pi(c)\cdot\pi(\omega)\geq 0$ where $\pi(c)$ is a circuit of $\textup{Cayley}(\A)$ restricted to columns indexed by the complement of $\I$ and $c$ is a circuit of $\textup{Cayley}(\A)$.
 We know that $c\cdot(-1_\I+\varepsilon\omega)\geq 0$. Therefore, using $\textup{supp}(c)\cap\I=\emptyset$ twice, we get
$$\pi(c)\cdot\pi(\omega)=c\cdot\omega={1\over\varepsilon}(\varepsilon c\cdot\omega)={1\over\varepsilon}(c\cdot (-1_\I)+\varepsilon c\cdot\omega)={1\over\varepsilon}(c\cdot (-1_\I+\varepsilon\omega))\geq 0.$$

``$\subseteq$'': Let $M$ be a mixed cell for $\A_{\I^c}$ with respect to $\pi(\omega)$. It is immediate that $M\cap\I=\emptyset$. Thus it remains to prove that $-1_\I+\varepsilon\omega$ is in the mixed cell cone of $M$. By Lemma~\ref{lem:ineq}, for every $\gamma\in\{1,\dots,m\}\setminus M$ there is one circuit inequality to check. For $\gamma\not\in \I$ the circuit $c$ is also a circuit $\pi(c)$ for $A_{\I^c}$. Hence $\pi(c)\cdot\pi(\omega)\geq 0$. We have $-1_\I\cdot c=0$ so $c\cdot(-1_\I+\varepsilon\omega)\geq 0$ follows. For $\gamma\in \I$ we must show that $c\cdot(-1_\I+\varepsilon\omega)\geq 0$. By Observation~\ref{obs:circuitsign}, $c_a< 0$. Because $\I$ and $\textup{supp}(c)$ overlap only at $a$ we get $-1_\I\cdot c>0$. But now follows, because $\varepsilon$ is small, that $c\cdot(-1_\I+\varepsilon)\geq 0$ as desired.
\end{proof}
While the lemma lets $\varepsilon>0$ go to zero, it is often more convenient to think about coefficients indexed by $\I$ going to $-\infty$, while the remaining are lifted by $\omega$. In the following subsections we apply this lemma.
\subsection{Total degree homotopy}
We describe the tropical analogue of the numerical total degree homotopy and start with an example of a set of mixed cells.
\begin{example}
\label{ex:startsystem}
Let $L_i$ denote the $n\times (n+1)$-matrix with columns $0,e_1,\dots,e_n$. The tuple $(L_1,\dots,L_n)$ of $n$ simplices has exactly one mixed cell w.r.t. the lift $\omega=(e_1+e_2)\times\cdots\times(e_1+e_{n+1})$, namely $M=((1,2), (1,3), \dots,(1,n+1))$. To see that this is a mixed cell, check that the lift $\omega-(1,\dots,1)$ gives this mixed cell following the argument in the proof of Lemma~\ref{lem:ineq}. Here $p=(0,\dots,0,1)$ and for each $L_i$, the upper face in direction $p$ indeed is the edge connecting the vertices of this configuration indexed by $M_i$. Because $(1,\dots,1)$ is in the row space of $\textup{Cayley}(L_1,\dots,L_n)$, also $\omega$ will induce this mixed cell of volume $1$. Computing the volume polynomial $\textup{vol}(\lambda_1L_1+\cdots+\lambda_nL_n)=\textup{vol}((\lambda_1+\cdots+\lambda_n)L_1)=(\lambda_1+\dots+\lambda_n)^n\textup{vol}(L_1)=(\lambda_1+\dots+\lambda_n)^n{1\over n!}$ we see that the coefficient of $\lambda_1\cdots\lambda_n$ is 1 and that there is just one mixed cell.
\end{example}

Consider $\A=(A_1,\dots,A_n)$ with all entries of all $A_i$ being in $\N=\Z_{\geq 0}$. By the total degree $\textup{deg}(A_i)$ of a configuration $A_i$ we mean the largest sum of the entries in a column of $A_i$. A priori, we do not know how to solve a tropical system with support $\A$, but if we for each $A_i$ append the exponent vectors $0,\textup{deg}(A_i)e_1,\dots,\textup{deg}(A_i)e_n$ as columns then a unique mixed cell of volume $\prod_i^{i=n}\textup{deg}(A_i)$ is induced by lifting the appended columns as in Example~\ref{ex:startsystem} and adding to the full vector of heights any $\varepsilon$-small perturbation.

The total degree homotopy now amounts to applying Algorithm~\ref{alg:homotopy} with the negated characteristic vector of the appended columns as a target vector. Thereafter Lemma~\ref{lem:breakingoff} is applied to obtain the solutions of the generic system.

\begin{definition}
For any vector $w\in\R^m$ and any ordering $\prec$ on $\R^n$ we define the new ordering $\prec_w$ by:
$$u\prec_w v\Leftrightarrow w\cdot u<w\cdot v\vee w\cdot u=w\cdot v\wedge u\prec v.$$
\end{definition}

Let $\prec$ be the lexicographic ordering on $\R^n$, or equivalently, for $\varepsilon>0$ sufficiently small, the vector $\varepsilon^0e_1+\varepsilon^1e_2+\cdots+\varepsilon^{m-1+n(n+1)}e_{m+n(n+1)}$.
\begin{algorithm}
[Tropical total degree homotopy]$ $\label{alg:totaldegreehomotopy}\\
{\bf Input:} A tuple $\A=(A_1,\dots,A_n)$ with $A_i\in\N^{n\times m_i}$.\\
{\bf Output:} The mixed cells of $\A$ for the lift $\prec$ defined above.
\begin{itemize}
\item For $i=1,\dots,n$ let $B_i$ be the $n\times(n+1)$-matrix $\textup{deg}(A_i)[0,e_1,\dots,e_n]$.
\item Let $\A':=([B_1A_1],\dots,[B_nA_n])$.
\item Let $S=\{((1,2),(1,3),\dots,(1,n+1))\}$.
\item Let $\tau$ be the vector with an entry for each column in $\A'$ having entries indexed by $B_i$ equal to $-1$ and the entries indexed by $A_i$ equal to $0$.
\item From $S$ compute the set $S'$ of mixed cells of $\A'$ with respect to $\prec_\tau$ via Algorithm~\ref{alg:homotopy}.
\item Let $S'':=\{M\in S':\textup{for all }i\textup{ } M_i\textup{ does not index a column of }B_i\}$.
\item Let $S$ be the vectors in $S''$ with $(n+1,n+1)$ subtracted from each pair.
\item Return $S$.
\end{itemize}
\end{algorithm}
We note that while a start system in the numerical total degree homotopy has segments as Newton polytopes, for the tropical algorithm they are simplices.
\subsection{Tropical regeneration}
As an alternative to the total degree homotopy, we mimic the numerical regeneration process~\cite{Hauenstein:2011:RHS}.
We start with $n$ tropical hyperplanes (i.e. tropical hypersurfaces defined by linear polynomials) in generic position having just one mixed cell. At step $i$ we scale the Newton polytope of the $i$th hyperplane, so that it contains the columns of $A_i$. Simultaneously scaling the coefficients, the solutions of the system are preserved. To the $n+1$ points in the support of the $i$th linear polynomial we then add the columns of $A_i$, the latter having low lifts. Moving the vertices of the scaled simplex to $-\infty$, only the columns $A_i$ remain in the subdivision.

\begin{example}
\label{ex:regeneration}
Suppose we want to solve a generic system with support given by matrices $A_1=\left(\begin{array}{ccc}
0 & 0 & 2 \\
0 & 1 & 0
\end{array}\right)
$ and $A_2=\left(\begin{array}{ccc}
0 & 0 & 1 \\
0 & 1 & 1
\end{array}\right)$. We notice that these fit inside $2\cdot\textup{conv}(0,e_1,e_2)$. In the regeneration process we start with two hypersurfaces in a position where the intersection is known. See Figure~\ref{fig:regeneration}. We then scale one simplex, and break off the excess pieces (see the first column of pictures). For $A_2$, the second simplex is first scaled, so that $A_2$ fits inside. Then excess vertices are moved to $-\infty$. After many combinatorial changes we obtain the solutions of a generic system of equations with support $(A_1,A_2)$. In particular, we can read off the mixed volume. 
\end{example}

\begin{figure}[b]
\begin{center}
\includegraphics[scale=0.19118]{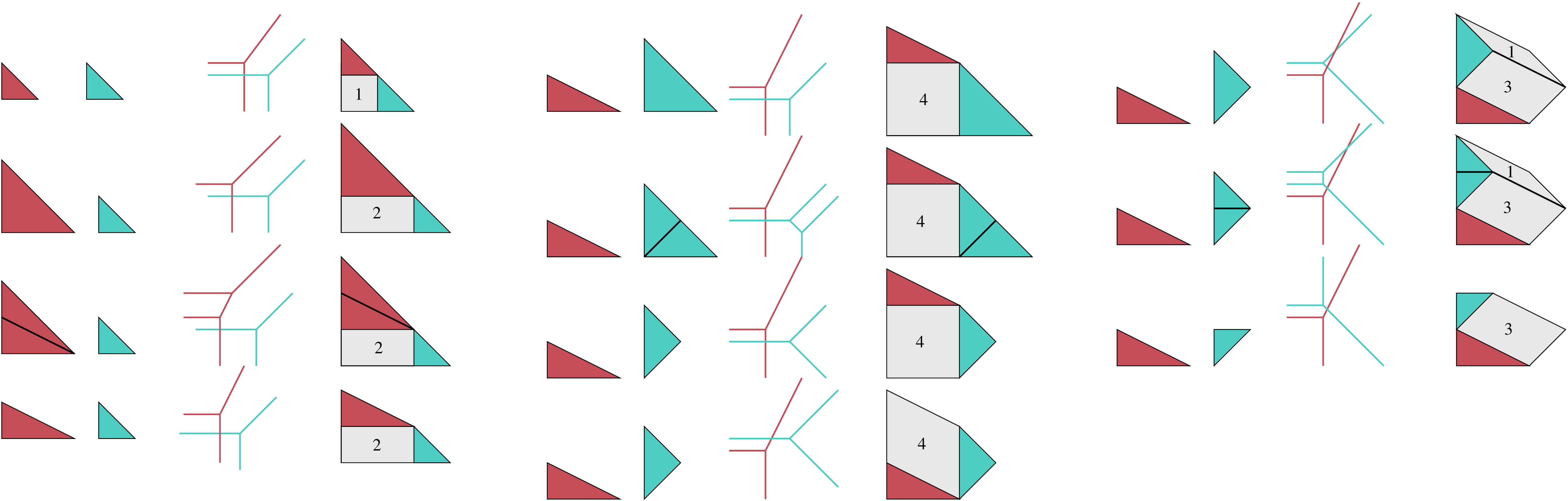}
\end{center}
\caption{Eleven steps in the tropical regeneration algorithm as described in Example~\ref{ex:regeneration}. Without being specific about the lift, for each step the subdivided Newton polytopes, the overlay of their associated hypersurfaces and the subdivision of their Minkowski sum is shown. To be read column by column.}
\label{fig:regeneration}
\end{figure}

It turns out to be practical for the required numerical precision to consider alternative lifts for the starting system. In particular the starting mixed cell will be different from what Figure~\ref{fig:regeneration} suggests. See also Remark~\ref{rem:scaling}.
\begin{lemma}
\label{lem:startsystem2}
If we change the lift in Example~\ref{ex:startsystem} to $(\varepsilon^0,\dots,\varepsilon^{n^2+n-1})$ for $\varepsilon>0$ sufficiently small, then there is only the mixed cell $((1,2),(2,3),\dots,(n,n+1))$.
\end{lemma}
\begin{proof}
Because the mixed volume is $1$ there is only one mixed cell. We will prove that this cell is  $((1,2),(2,3),\dots,(n,n+1))$ by proving that the lift satisfies the circuit inequality for each choice of additional Cayley column index $\gamma$. Let $c$ be the induced circuit with $c_\gamma=-1$. It suffices to prove that the first non-zero entry of $c$ is positive. The proof splits into three cases. Let $e_{a,b}$ denote the unit vector with the entry of the $b$th point in the $a$th configuration equal to $1$.

If the index $\gamma$ is chosen from the $i$th configuration (in which $(i,i+1)$ are chosen for the cell) with $\gamma>i+1$ then by the uniqueness of $c$, we have
$c=e_{i,i+1}-e_{i,\gamma}-e_{i+1,i+1}+e_{i+1,i+2}-\cdots-e_{\gamma-1,\gamma-1}+e_{\gamma-1,\gamma}$ as this vector is in the null space of $\textup{Cayley}(L_1,\dots,L_n)$. It has first non-zero entry positive.

If the index $\gamma>1$ is chosen from the $i$th configuration with $\gamma<i$ then by the uniqueness of $c$, we have $c=e_{\gamma,\gamma}-e_{\gamma,\gamma+1}+e_{\gamma+1,\gamma+1}-e_{\gamma+1,\gamma+2}+\cdots +e_{i-1,i-1}-e_{i-1,i}-e_{i,\gamma}+e_{i,i}$. It has first non-zero entry positive.

If the index $\gamma=1$ is chosen from the $i$th configuration then by the uniqueness of $c$ we have $c=e_{1,1}-e_{1,2}+e_{2,2}-e_{2,3}+\cdots+e_{i-1,i-1}-e_{i-1,i}-e_{i,1}+e_{i,i}$. It has first non-zero entry positive.
\end{proof}
Let again $\prec$ denote the lexicographic ordering on $\R^m$ for any $m$, i.e. the ordering represented by $(\varepsilon^0,\dots,\varepsilon^{m-1})$ with $\varepsilon>0$ infinitesimally small.
\begin{algorithm}
[Tropical regeneration]$ $\label{alg:regeneration}\\
{\bf Input:} A tuple $\A=(A_1,\dots,A_n)$ with $A_i\in\N^{n\times m_i}$.\\
{\bf Output:} The mixed cells of $\A$ for the generic lift $\prec$.
\begin{itemize}
\item For $i=1,\dots,n$ let $B_i$ be the $n\times(n+1)$-matrix $\textup{deg}(A_i)[0,e_1,\dots,e_n]$ and let $L_1=\cdots=L_n$ denote the Newton polytope of a linear polynomial.
\item Let $S:=\{((1,2),(2,3),\dots,(n,n+1))\}$.
\item For $i=1,\dots,n:$
\begin{itemize}
\item Invariant: $S$ is the set of mixed cells of $(A_1,\dots,A_{i-1},L_{i},\dots,L_n)$ with respect to $\prec$.
\item Consequently, $S$ is the set of mixed cells with respect to $\prec$ of $\A':=(A_1,\dots,A_{i-1},[B_iA_i],L_{i+1},\dots,L_n)$
\item Let $\tau$ be the vector with an entry for each column in $\A'$ having entries indexed by $B_i$ equal to $-1$ and all other entries zero.
\item Compute the set $S'$ of mixed cells of $\A'$ with respect to $\prec_\tau$ via Algorithm~\ref{alg:homotopy}.
\item Let $S'':=\{M\in S':\textup{for all } i\textup{ } M_i\textup{ does not index any column of }B_i\}$.
\item Let $S$ be the vectors in $S''$ with $(n+1,n+1)$ subtracted from their $i$th pair.
\end{itemize}
\item Return $S$.
\end{itemize}
\end{algorithm}
\begin{proof}
We will prove the correctness of the algorithm by proving that the invariant is satisfied at each iteration. The base case follows from Lemma~\ref{lem:startsystem2}, and the consequence stated in the algorithm follows from the columns of $B_i$ being lifted low. For the induction step Lemma~\ref{lem:breakingoff} is applied.
\end{proof}

We refer to Remark~\ref{rem:scaling} for a discussion of the reasons for the choice of $L_i$ over $B_i$ and the ordering $\prec$ over the start vector proposed in Example~\ref{ex:startsystem}.
\begin{example}
Consider the cyclic 10-roots system $f_1,\dots,f_{10}\in \CC[x_1,\dots,x_{10}]$, where $f_i$ has $10$ terms and $\textup{deg}(f_i)=i$ for $i=1,\dots,9$ while $f_{10}=x_1\cdots x_{10}-1$. To do the tropical regeneration we need to setup 10 homotopies.
The first of these homotopies works on a Cayley matrix with $10+10\cdot 11=120$ columns, while the last Cayley matrix has $10\cdot 9+2+11=103$ columns.
\end{example}
\begin{remark}
When applying Algorithm~\ref{alg:homotopy} in Algorithm~\ref{alg:regeneration}, we note that in many cases only circuit walls arising when the additional column $\gamma$ is picked inside $[B_iA_i]$ need to be considered. This is because choosing a Cayley column $\gamma$ outside $[B_iA_i]$, the circuit $c$ restricted to $[B_iA_i]$ is forced to have (none or) two non-zero entries adding to zero. As a consequence if $\tau$ has the same values for these entries, then $\langle c,\tau\rangle=0$ and $c$ can be excluded from our considerations by the argument in Section~\ref{sec:generic}.
\end{remark}

\section{Malajovich' method}
\label{sec:malajovich}
In \cite{malajovich} Malajovich proposes to compute the mixed volume of $n$ polytopes by intersecting $n$ tropical hypersurfaces in a way somewhat similar to Algorithm~\ref{alg:regeneration}. The key observation is that if coefficients are generic and one of the tropical hypersurfaces is ignored, then the intersection of the remaining tropical hypersurfaces is a graph embedded in $\R^n$. This graph may or may not be connected. To find the connected components, Malajovich
computes the intersection of the graph with a generic classical
hyperplane intersecting all components. This computation is a problem of the same kind in lower dimension. By recursion the intersection is computed, all components of the embedded graph are traversed, and the result is finally intersected with the ignored tropical hypersurface.

At a first glance it
 seems that our choice of a tropical hyperplane rather than a classical one
will only lead to minor differences between the two algorithms.
However, in~\cite{malajovich} bringing the hypersurfaces in general position and choosing a generic classical hyperplane is done by picking random floating point coefficients. As a consequence subsequent computations must be carried out in floating point arithmetic. Later branching in the algorithm will depend on these computations. Attempts are made to predict the required precision and therefore the method is unlikely to fail due to inconsistent round off. However, it is not an exact method and our symbolic perturbations are quite different.

Besides exactness, our contribution is that of applying reverse search. The advantages are two-fold. On one hand parallelisation is easier and on the other memory usage, which was reported as a problem in~\cite{malajovich}, is lower. Reverse search
could also be applied in the setting of Malajovich, but again correctness and possibly termination would depend on floating point approximations.

A major theoretical contribution by Malajovich is
the time complexity bound on his algorithm in terms of  
quermassintegrals i.e. in terms of the geometry of the setting.
This makes his method radically different from earlier methods like~\cite{li,Mizutani2007}, where the enumeration tree has limited geometric meaning.

Let $A_1,\dots,A_n$ and $\omega_1,\dots,\omega_n$ be fixed with $\omega_1,\dots,\omega_n$ generic. Malajovich bounds the complexity of his algorithm by bounding the number of edges in $T(A_1,\omega_1)\wedge\cdots\wedge T(A_{n-1},\omega_{n-1})$, while also making estimates for the number of edges in the recursion.
The following theorem is essentially a reformulation of his result
in our setting.
\begin{theorem}
The number of edges in $T(A_1,\omega_1)\wedge\cdots\wedge T(A_{n-1},\omega_{n-1})$ is at most $3\cdot n!\cdot\textup{MixVol}(\textup{conv}(A_1),\dots,\textup{conv}(A_{n-1}),\sum_{i=1}^{n-1}\textup{conv}(A_i))$ under the assumption that $\sum_{i=1}^{n-1}\textup{conv}(A_i)$ is full-dimensional and $\omega_1,\dots,\omega_{n-1}$ are generic.
\end{theorem}
\begin{proof}
To simplify notation, define $A_n$ as the $n\times 1$ zero matrix and $\omega_n=0\in\R^1$.
The edges that we count are dual to cells in the regular mixed subdivision of $A_1,\dots,A_{n}$ induced by $\omega_1,\dots,\omega_n$. By the type of a cell $Z$ we mean the vector $v\in\N^n$ where $v_i$ is the dimension of the $i$th summand of $Z$. By genericity of $\omega$, we have that $\sum_i v_i=n$. Each full-dimensional cell contributes to a term of the volume polynomial of $\textup{conv}(A_1),\dots,\textup{conv}(A_n)$. We are interested in counting the cells of type $(1,\dots,1,0)$. They are facets of cells of type $(1,\dots,1,0)+e_i$ with $i\in\{1,\dots,n-1\}$. In the volume polynomial of $\textup{conv}(A_1),\dots,\textup{conv}(A_n)$ they contribute to the term $\lambda_1\cdots\lambda_{n-1}\lambda_i$. By basic properties of mixed volume the coefficient of this term is $\textup{MixVol}(\textup{conv}(A_1),\dots,\textup{conv}(A_{n-1}),\textup{conv}(A_i))$. Since each cell can contribute at most its volume times $3n!$ edges (3 arising because the summand from $A_i$ is two-dimensional), we get by multilinearity
that the total number of edges in $T(A_1,\omega_1)\cap\cdots\cap T(A_{n-1},\omega_{n-1})$ is bounded by $$3\cdot n!\cdot(\textup{MixVol}(\textup{conv}(A_1),\dots,\textup{conv}(A_{n-1}),\textup{conv}(A_1))+\cdots$$ $$+\textup{MixVol}(\textup{conv}(A_1),\dots,\textup{conv}(A_{n-1}),\textup{conv}(A_{n-1})))=$$ $$3\cdot n!\cdot\textup{MixVol}(\textup{conv}(A_1),\dots,\textup{conv}(A_{n-1}),\sum_{i=1}^{n-1}\textup{conv}(A_i)).$$
\end{proof}
Successively replacing each $A_i$ by the Newton polytope of an affine function, we obtain upper bounds for the number of bistellar flips in our algorithm. Notice, however, that it is possible for our algorithm to move along the same edge more than once in the $i$th iteration. This is because higher degree hypersurfaces can intersect a tropical curve in several points, which then may move along the same edge as the hypersurface is deformed. Therefore our bound must be multiplied by, say,
$m_i\choose 2$. This is the price we pay for not actually storing the graph, but rather having a \emph{memoryless} algorithm. It is not much worse than the method of~\cite{malajovich}, since that needs to intersect the embedded graph with $T(A_n,\omega_n)$.

\section{Solving non-generic systems}
\label{sec:nongensolving}
Until now we have been interested in tropical square systems with generic coefficients. This has the advantage that all solutions are isolated points and we have seen that these points can be found via homotopy methods. For numerical systems with non-generic coefficients it has been known at least since~\cite{Morgan:1987} that the total degree homotopy will find all \emph{isolated} solutions of the target system. We will prove a similar statement for tropical homotopies. Our proof will use notions (\emph{balanced, weighted and pure fans, links, stable intersection ($\cap_\textup{st}$)}) and properties of stable intersection which for example are described in~\cite{tropStab}. Proposition~\ref{prop:intersectionbound} below is a generalisation of~\cite[Proposition 3.2.1]{ossermanpayne} to any number of fans.
By the codimension of a fan $\F$ in $\R^n$ we mean $n-\textup{dim}(\F)$.

\begin{lemma}
\label{lem:intersectionbound}
Let $\F$ be a balanced fan in $\R^n$ with positive weights and let $L$ be a rational linear subspace of $\R^n$. Then $\textup{codim}(\F\cap L)\leq \textup{codim}(\F)+\textup{codim}(L)$.
\end{lemma}
\begin{proof}
Let $R=\textup{span}(\F+L)$. Let $\LL\subseteq L$ be a linear subspace of dimension $\textup{dim}(R)-\textup{dim}(\F)$ such that $\textup{dim}(\F+\LL)=\textup{dim}(R)$. We have $(\F\cap L)+\LL=(\F+\LL)\cap(L+ \LL)=R\cap L=L$ 
and get $\textup{dim}(\F\cap L)\geq\textup{dim}(L)-\textup{dim}(\LL)$. This implies $\textup{codim}(\F\cap L)=n-\textup{dim}(\F\cap L)\leq n-\textup{dim}(L)+\textup{dim}(\LL)= n-\textup{dim}(L)+\textup{dim}(R)-\textup{dim}(\F)\leq n-\textup{dim}(L)+n-\textup{dim}(\F)=\textup{codim}(\F)+\textup{codim}(L)$. We used the positivity and balancing when constructing $\LL$ and claiming $\F+\LL=R$.
\end{proof}

\begin{proposition}
\label{prop:intersectionbound}
Let $\F_1,\dots,\F_k$ be balanced fans in $\R^n$ with positive weights. Then $\textup{codim}(\F_1\cap\cdots\cap\F_k)\leq \textup{codim}(\F_1)+\cdots+\textup{codim}(\F_k)$.
\end{proposition}
\begin{proof}
We note that $\F:=\F_1\times\cdots\times\F_k$ is a balanced fan and conclude from Lemma~\ref{lem:intersectionbound} that $\textup{codim}(\F\cap L)\leq\textup{codim}(\F)+\textup{codim}(L)$ where $L$ is the $n$-dimensional diagonal $\{(a,\dots,a):a\in\R^n\}\subseteq(\R^n)^k$. Hence
$\textup{codim}(\F_1\cap\cdots\cap\F_k)=n-\textup{dim}(\F_1\cap\cdots\cap\F_k)=n-\textup{dim}(\F\cap L)=\textup{codim}(\F\cap L)-(k-1)n\leq\textup{codim}(\F)+\textup{codim}(L)-(k-1)n=\textup{codim}(\F)=\sum_i\textup{codim}(\F_i)$.
\end{proof}
The proposition is most useful if applied to links in fans.
\begin{corollary}
For pure balanced fans $\F_1,\dots,\F_k$ with positive weights in $\R^n$ and $\omega\in\textup{supp}(\F_1\wedge\cdots\wedge\F_k)$:
$$\textup{dim}(\textup{link}_\omega(\F_1\cap\cdots\cap\F_k))\geq n-\sum_i\textup{codim}(\F_i).$$
\end{corollary}
Here we can think of the dimension of the link as a local dimension around $\omega$.
\begin{proof}
Using the definition of codimension, the inequality $\textup{codim}(\textup{link}_\omega(\F_1\cap\cdots\cap\F_k))=\textup{codim}(\textup{link}_\omega(\F_1)\cap\cdots\cap\textup{link}_\omega(\F_k))\leq \textup{codim}(\textup{link}_\omega(\F_1))+\cdots+\textup{codim}(\textup{link}_\omega(\F_k))=\textup{codim}(\F_1)+\cdots+\textup{codim}(\F_k)$ implies the result.
\end{proof}
\begin{lemma}
\label{lem:mixvolchar}
Let $P_1,\dots,P_n$ be convex polytopes in $\R^n$.
The mixed volume $\textup{MixVol}(P_1,\dots,P_n)$ is zero if and only if there exists a subset $I\subseteq\{1,\dots,n\}$ such that $\textup{dim}(\sum_{i\in I}P_i)<|I|$.
\end{lemma}
While a proof using the BKK theorem and sparse
resultant varieties appeared after Theorem~2.29 in~\cite{tropRes}, we shorten the proof and avoid algebraic geometry. 
\begin{proof}
Since mixed cell candidates have non-empty mixed cell cones, the mixed volume is non-zero if and only if $(P_1,\dots,P_n)$ has a mixed cell candidate. Let $C_i=\{a-b:a,b\textup{ are vertices of }P_i\}$. By applying Rado's generalisation of Hall's Theorem~\cite[Theorem~1]{rado} to $(C_1,\dots,C_n)$, we get that a mixed cell candidate exists if and only if for all $I\subseteq\{1,\dots,n\}:\textup{dim}(\sum_{i\in I}\textup{span}(C_i))\geq|I|$.
\end{proof}

\begin{lemma}
\label{lem:nonemptystable}
For balanced tropical fans $\F_1$ and $\F_2$ in $\R^n$ with positive weights the stable intersection $\F_1\cap_{st}\F_2$ is non-empty if and only if there exist facets $C_1$ of $\F_1$ and $C_2$ of $\F_2$ such that $\textup{dim}(C_1+C_2)=n$.
\end{lemma}
\begin{proof}
The lemma follows from~\cite[Definition~2.4]{tropStab} and~\cite[Corollary~2.3]{tropStab}.
\end{proof}
By the tropical hypersurface $T(P)$ of a lattice polytope $P$ we mean the hypersur\-face of any tropical polynomial with Newton polytope $P$ and coefficients $0$.
\begin{theorem}
\label{thm:impliesstablezero}
For lattice polytopes $P_1,\dots,P_n$ in $\R^n$, if $T(P_1)\cap\cdots\cap T(P_n)=\{0\}$ then $\textup{MixVol}(P_1,\dots,P_n)\not=0$.
\end{theorem}
\begin{proof}
Let $L_i$ denote the affine span of $P_i$ translated to the origin.
Suppose $\textup{MixVol}(P_1,\dots,P_n)=0$. Then by Lemma~\ref{lem:mixvolchar} there exists a selection $I\subseteq\{1,\dots,n\}$ such that $\textup{dim}(\sum_{i\in I}L_i)<|I|$. Without loss of generality $I=\{1,\dots,d\}$. Define $V:=\bigcap_{i\in I}L_i^\perp\subseteq \bigcap_{i\in I}T(P_i)$ of codimension strictly less than $d$. The remaining $P_{d+1},\dots,P_n$ are
assumed ordered such that $V\subseteq T(P_i)$ for $i=d+1,\dots,D$ and $V\not\subseteq T(P_i)$ for $i=D+1,\dots,n$. Hence $V\subseteq\sum_{i=1}^D T(P_i)$.

We argue that for $i\in\{D+1,\dots,n\}$, we have $V\cap_\textup{st} T(P_i)\not=\emptyset$.
If we can find a facet $F$ of $T(P_i)$ such that $V\not\subseteq\textup{span}(F)$ then we are done by Lemma~\ref{lem:nonemptystable}. Suppose for contradiction that for every facet $F$ we have $V\subseteq\textup{span}(F)$. Then $V$ is perpendicular to each edge of $P_i$ and therefore perpendicular to $P_i$. We get $V\subseteq T(P_i)$ --- contradicting our earlier assumptions.

Therefore, by additivity of codimension for stable intersection, the stable intersections $V\cap_\textup{st}T(P_{D+1}),\dots,V\cap_\textup{st}T(P_{n})$ have codimension $1$ \emph{inside} $V$.

By Proposition~\ref{prop:intersectionbound} the intersection
$(V\cap_\textup{st}T(P_{D+1}))\cap\cdots\cap(V\cap_\textup{st}T(P_{n}))$
has codimension at most $n-D$ in $V$, which again has codimension strictly less than $d\leq D$ in $\R^n$. Consequently $\textup{codim}((V\cap_\textup{st}T(P_{D+1}))\cap\cdots\cap(V\cap_\textup{st}T(P_{n})))<(n-D)+D=n$ and therefore this intersection must contain a non-zero point $u\in\R^n$. For $i=1,\dots,D$ we have $u\in V\subseteq T(P_i)$. Because $u$ is in the intersection above, we also have $u\in T(P_{D+1})\cap\cdots\cap T(P_n)$. Consequently, $u\in T(P_1)\cap\cdots\cap T(P_n)$, implying $T(P_1)\cap\cdots\cap T(P_n)\not=\{0\}$ as desired.
\end{proof}

\begin{theorem}
If we are given a system $\A=(A_1,\dots,A_n)$ with coefficients $\omega=\omega_1\times\cdots\times \omega_n\in\R^m$, an isolated $p\in T(A_1,\omega_1)\cap\cdots\cap T(A_n,\omega_n)$ and generic $\omega'\in\R^m$, then for every $\varepsilon>0$ there exists $\delta>0$ such that for every $t\in (0,\delta)$:
$$T(A_1,\omega_1+t \omega_1')\cap\cdots\cap T(A_n,\omega_n+t \omega_n')\cap B(p,\varepsilon)\not=\emptyset$$ where $B(p,\varepsilon)\subseteq\R^n$ is the open ball centered at $p$ with radius $\varepsilon$.
\end{theorem}
\begin{proof}
  Let $P_i$ denote the projection $\pi(\textup{face}_{p_i\times\{1\}}(\textup{conv}_j((A_i)_j\times((\omega_i)_j))))$. Then $\sum_iP_i$ is the dual cell of $p$ in the mixed subdivision of $A_1,\dots,A_n$ induced by $\omega$. Because the link of $\bigcap_iT(A_i,\omega_i)$ at $p$ equals $T(P_1)\wedge\dots\wedge T(P_n)$ and $p$ is isolated we get that $\textup{MixVol}(P_1,\dots,P_n)\not=0$ by Theorem~\ref{thm:impliesstablezero}. Hence the mixed subdivision of $P_1,\dots,P_n$ induced by the restriction $\omega''$ of $\omega'$ to the vertices of $P_1,\dots,P_n$ has a fully mixed cell $Z$. Note that for $t>0$ sufficiently small, the mixed subdivision of $\A$ induced by $\omega+t\omega'$ is a refinement of that induced by $\omega$ and the fully mixed cell $Z$ appears in it. The coordinates of the solution dual to $Z$ are obtained continuously from the upper normal of the lift of $Z$. Therefore for sufficiently small $t$, this solution to $T(A_1,\omega_1+t \omega_1')\cap\cdots\cap T(A_n,\omega_n+t \omega_n')$ is $\varepsilon$-close to $p$.
\end{proof}
The proof shows that the isolated solution is obtained from the mixed cells with respect to a perturbed lift $\prec_\omega$ simply by computing the normal (for the lift $\omega$) with last coordinate $1$ of each lifted cell and projecting it to $n$ dimensions.

\begin{algorithm}
[Non-generic system solving]$ $\label{alg:nongeneric}\\
{\bf Input:} A tuple $\A=(A_1,\dots,A_n)$ with $A_i\in\N^{n\times m_i}$, the mixed cells $S$ of $\A$ for a generic lift $\prec$ and a vector $\omega\in\R^{m_1}\times\cdots\times\R^{m_n}$.\\
{\bf Output:} A finite superset $P$ of the isolated points in $T(A_1,\omega_1)\cap\cdots\cap T(A_n,\omega_n)$.
\begin{itemize}
\item Apply Algorithm~\ref{alg:homotopy} to $\A,S,\prec$ and $\tau:=\omega$ to obtain $S'$.
\item Let $P=\emptyset$.
\item For $M\in S'$:
\begin{itemize}
\item Compute the hyperplane intersection $$\{p\}:=T((A_1)_{M_1},(\omega_1)_{M_1})\cap\cdots\cap T((A_n)_{M_n},(\omega_n)_{M_n})$$ where the subscripts denote the restriction to columns and entries indexed by $M_i$.
\item  $P:=P\cup\{p\}$.
\end{itemize}
\item Return $P$.
\end{itemize}
\end{algorithm}
Deciding if each of the produced solutions is isolated is a matter of deciding if $\bigcap_i\textup{link}_p(T(A_i,\omega_i))=\{0\}$. The following problem is in NP, but is it NP-hard?
\begin{itemize}
\item Given $(A_1,\dots,A_n)\in \N^{n\times m_1}\times\cdots\times\N^{n\times m_n}$, decide if $\bigcap_iT(A_i,0)\not=\{0\}$.
\end{itemize}
This does not seem to be an immediate consequence of the results in~\cite{theobald}.

We encourage the reader to compare Theorem~\ref{thm:impliesstablezero} with \cite[Corollary 5.2.4]{ossermanpayne}, which relates the mixed volume to the number of Puiseux series solutions with a particular valuation to a polynomial system with Puiseux series coefficients.

\section{Implementation and experiments}
\label{sec:experiments}

We have implemented Algorithm~\ref{alg:regeneration} in C++ as a part of the command line computer algebra system \textup{gfan}, which is specialised in Gr\"obner fan computations and tropical geometry. The implementation relies on the GNU multiprecision library \cite{gmp} and for parallelisation on an abstract tree traversal C++11 library contributed to the gfan project by Bjarne Knudsen. No linear programming solver is used in the algorithm.

As described in Section~\ref{sec:generic} we follow a symbolically perturbed line in $\R^m$ when performing tropical homotopies. In particular no random floating point numbers need to be generated. Floating point numbers are however used when finding circuit inequalities as generators for nullspaces of matrices. These computations are always checked in exact machine arithmetic afterwards. If the check fails, the implementation falls back on exact GMP arithmetic for computing the nullspace. If the primitive generator for the nullspace has entries which do not fit in 32 bits, the whole mixed cell enumeration fails. Similarly, entries of the input $A_1,\dots,A_n$ must fit in signed 16 bit words for our implementation. The first restriction causes one test example (Gaukwa 9) to fail.

We describe how the traversal problem is given to the abstract parallel tree traverser. To traverse a tree using $k$ threads, $k$ \emph{traverser objects} are created and placed at the root of the tree. Each traverser must supply methods for computing the number of children at its current vertex, moving to its $i$th child and moving one step up in the tree. The tree traversal library then takes care of moving the traversers around so that every leaf is computed. When implementing Algorithm~\ref{alg:regeneration} we combine the $n$ steps into a single tree rather than doing several forest traversals. In this almost memoryless implementation only the $k$ paths from the root to the vertices of the traversers are stored.

Two performance improvements of the implementation still remain to be done. One is the introduction of \emph{rank-1 updates} of the inverse matrix of the Cayley submatrix indexed by the mixed cell. Rather, at the moment, an $n\times n$ floating point matrix is inverted for each mixed cell under consideration. The other improvement is representing circuit inequalities \emph{sparsely}.

 We ran experiments on a system with two Intel Xeon E2670 CPUs, each with 8 cores and each core supporting hyperthreading. We chose to run with at most 16 threads as double speed cannot be expected with hyperthreading.
Doubling speed can also not be expected when going from 8 to 16 threads as some cores will share caches and bus access. Moreover, the Intel Turbo Boost technology makes it impossible to achieve linear speed-up when increasing the number of threads, as it allows dynamically changing clock frequency based on factors such as temperature.
Indeed the clock frequencies 3.3 and 3.0 GHz were typically observed for 1 and 16 threads respectively, making $14.55$ the largest possible theoretical speed-up factor when going from 1 to 16 threads.

We chose to run our software on the example classes appearing in~\cite[Table~3]{malajovich}. Most of these classes were also tested in~\cite{leeli}. The results for 1 and 16 threads are shown in the table in Figure~\ref{fig:timings} together with the timings from~\cite{leeli} (``2.4GHz Intel Core 2 Quad CPU'') and~\cite{malajovich} (``SGI Altix ICE 8400''). With the timings of~\cite{leeli} being outdated, we list the few timings reported in the newer article~\cite{li2014} for the single threaded implementation on unspecified hardware: Cyclic-15: 8.4h, Eco-20: 3.1h, Katsura-15: 48m, Noon-21: 54m.

From the table it is hard to scientifically draw general conclusions about the relative performance of the algorithms and their implementations across the example classes --- one reason being the different computer architectures. We hope to run the software of~\cite{malajovich} and~\cite{leeli} on our test machine in the future.

\begin{figure}
{
\begin{tabular}{|l|rrrr|}
\hline
Problem&n&Mixed vol & 1 thread & 16 thr.\\
\hline
Cyclic10&10&35940&5.1&0.6\\
Cyclic11&11&184756&32.0&2.7\\
Cyclic12&12&500352&152.4&11.7\\
Cyclic13&13&2704156&998.2&73.3\\
Cyclic14&14&8795976&4999.0&366.2\\
Cyclic15&15&35243520& &2017.3\\
Cyclic16&16&135555072& &10151.2\\
\hline
Noon16&16&43046689&86.3&7.5\\
Noon17&17&129140129&217.4&17.1\\
Noon18&18&387420453&540.5&41.4\\
Noon19&19&1162261429&1342.1&98.4\\
Noon20&20&3486784361&3268.4&243.5\\
Noon21&21&10460353161&7922.0&581.6\\
Noon22&22&31381059565& &1390.7\\
Noon23&23&94143178781& &3256.4\\
\hline
Chandra15&15&16384&38.9&3.9\\
Chandra16&16&32768&98.2&8.1\\
Chandra17&17&65536&246.8&18.3\\
Chandra18&18&131072&610.8&46.5\\
Chandra19&19&262144&1520.3&113.1\\
Chandra20&20&524288&3709.9&274.0\\
Chandra21&21&1048576&8951.9&656.9\\
Chandra22&22&2097152& &1574.6\\
Chandra23&23&4194304& &3664.9\\
\hline
Katsura15&16&32730&31.2&3.8\\
Katsura16&17&65280&76.4&7.4\\
Katsura17&18&131070&182.2&15.8\\
Katsura18&19&261576&449.4&37.1\\
Katsura19&20&524286&1049.7&79.9\\
Katsura20&21&1047540&2526.4&190.1\\
Katsura21&22&2097018&5807.5&428.4\\
Katsura22&23&4192254& &961.1\\
Katsura23&24&8388606& &2194.5\\
\hline
Gaukwa5&10&14641&1.0&0.2\\
Gaukwa6&12&371293&22.0&1.9\\
Gaukwa7&14&11390625&520.3&38.9\\
Gaukwa8&16&410338673&13627.7&1008.3\\
Gaukwa9&18& &  &  \\
\hline
Eco19&19&131072&648.3&51.0\\
Eco20&20&262144&1500.4&115.1\\
Eco21&21&524288&3472.8&259.1\\
Eco22&22&1048576&7962.8&585.8\\
Eco23&23&2097152&18248.6&1289.6\\
Eco24&24&4194304& &3021.4\\
Eco25&25&8388608& &6594.2\\
\hline
\end{tabular}
\begin{tabular}{|rr|}
\hline
Mal. 8t.& Lee,Li 1t.\\
\hline
& \\
& \\
39.5& 57\\
206& 504\\
850& 4034\\
4070& 36428\\
& \\
\hline
& \\
& \\
1230& \\
2870& 635\\
6460& 1109\\
& 4302\\
& 9214\\
& 24265\\
\hline
& \\
& \\
& \\
518& \\
1270& \\
3080& 462\\
7580& 1067\\
& 2601\\
& 7381\\
\hline
557& 2570\\
1880& 14561\\
5310& 75619\\
14200& \\
(300)& \\
& \\
& \\
& \\
& \\
\hline
& \\
& \\
70& 275\\
1020& 10702\\
& 370099\\
\hline
928& \\
1930& \\
4620& \\
8750& \\
& \\
& \\
& \\
\hline
\end{tabular}
}
\caption{Timings in seconds for Algorithm~\ref{alg:regeneration} compared to~\cite{malajovich} and \cite{leeli}.}
\label{fig:timings}
\end{figure}

It is nevertheless worthwhile to make some observations.
\begin{itemize}
\item Gfan performs better on the Katsura class than its competitors.  The speed-up is two orders of magnitude against both competitors on some examples. (The number $(300)$ refers to the heuristic method used in~\cite{malajovich}.)
\item The asymptotic behaviour in each of the families Cyclic and Eco seems worse than that of~\cite{malajovich}. For the Chandra examples it is similar.
\item Gfan gets behind~\cite{leeli} by an order of magnitude in some Chandra examples and almost an order of magnitude behind~\cite{malajovich} in the Cyclic examples.
\end{itemize}
We conclude that the method is competitive for the above reasons.
\begin{itemize}
\item The speed-up factor going from 1 to 16 threads is typically in the range 12-14 with a factor of 14.2 obtained at Chandra 21.
\item Additional statistics produced by the program reveals that the cast from a floating point to an integral vector, and hence the fall back on GMP numbers, only appears in the Gaukwa examples. 
\end{itemize}

\begin{remark}
\label{rem:scaling}
The choice of start orderings in Algorithm~\ref{alg:regeneration} affects the amount of work to be done. On some examples, indeed the choice of Example~\ref{ex:startsystem} and Algorithm~\ref{alg:totaldegreehomotopy} is better, but this depends on the example family. Better timings for the Chandra and Katsura classes can be obtained in this way. There is however a particular reason that we have chosen to use an ordering not being the refinement of a vector with full support. If we did that, then there would be no guarantee that we could carry over the set of mixed cells when we replace an $L_i$ by a $B_i$. Therefore we would have to work with $B_i$ all the time. We tried this and it affected the precision needed in the code. Working with $B_i$ gave considerably more 32-bit overflows that had to be handled with GMP integers. This caused congestion for the parallel implementation at heap allocations and the general performance got worse. Another consequence of using $B_i$ instead of $L_i$ was that the failure in the Gaukwa class appeared already for Gaukwa 8.
\end{remark}

To conclude, the main contribution of the tropical homotopy continuation algorithm to mixed cells enumeration is its exactness and its application of reverse search. Our experiments show that most arithmetic can be handled with machine precision. Moreover, the reverse search allows either a memoryless traversal or a parallelisation as a tree traversal with good scaling properties.

\noindent

\section{Future directions and open problems}
We finish this article by listing some questions for future research in tropical polynomial system solving.
\begin{itemize}
\item What is the complexity of deciding if a solution to a square system is isolated?
\item Is it possible to extend our methods to overdetermined systems?
\item Is it possible to find higher-dimensional solution components with tropical homotopy continuation?
\item Is there an output sensitive algorithm for finding the mixed cells?
\item How is the proposed method best combined with a numerical solver, i.e. which lifts will be convenient for both tropical and numerical homotopy?
\end{itemize}

\newpage
\appendix
\bibliographystyle{hplain}
\bibliography{jensen}
\end{document}